\documentclass[]{article}
\usepackage[margin=1in]{geometry}
\usepackage[utf8]{inputenc}
\usepackage[english]{babel}
\usepackage{amsmath,amssymb,amsfonts,amsthm}
\usepackage{amsfonts}
\usepackage{graphics}
\usepackage{graphicx}
\usepackage{subfigure}
\usepackage{cancel}
\usepackage{mathtools}
\usepackage{subfig}
\usepackage{xcolor}

\def \Re {\mathbb{R}}
\DeclarePairedDelimiter\floor{\lfloor}{\rfloor}

\newtheorem{theorem}{Theorem}[section]

\newtheorem{example}{Example}[section]

\newtheorem{proposition}[theorem]{Proposition}
\newtheorem{remark}[theorem]{Remark}

\title{A Finite Elements Strategy for Spread Contract Valuation Via  Associated  PIDE }
\author{Pablo Olivares and Ciro Diaz}
\begin{document}

\maketitle

\begin{abstract}
We study an efficient strategy based on finite elements to value spread options on commodities whose underlying assets follow a dynamic described by a certain class of two-dimensional Levy models by  solving  their associated  partial integro-differential equation (PIDE). To this end we consider a Galerkin approximation in space along with an implicit $\theta$-scheme for time evolution. Diffusion and drift in the associated operator are discretized using an exact Gaussian quadrature, while the integral part corresponding to jumps is approximated using the \textit{symbol method} introduced in \cite{gass2018flexible}. A system with blocked Toeplitz with Toeplitz blocks (BTTB) matrix is efficiently solved via biconjugate stabilized gradient method (BICSTAB)  with a circulant pre-conditioner at each time step. The technique is applied to the pricing of \textit{crack} spread options between the prices of futures RBOB gasoline (reformulated blendstock for oxygenate blending) and West Texas Intermediate(WTI) oil in NYMEX.

\end{abstract}

\section{Introduction}

Spreads contracts on commodities are the difference
between the finished product and the raw material used 
in its production. Examples of these contracts are  \textit{cracks}, based  on the difference between the 
price of gasoline  and oil, \textit{sparks}, based on 
the difference between electricity and natural gas prices and \textit{crushs} based on the difference between soy oil and soy beans. Typically, manufactures look at the difference between a long position on the raw material  and a short position on the resulting product.
\par
The goal of this paper is to investigate a pricing approach following the numerical solution of the  PIDE associated to option prices driven by two-dimensional Levy models, using finite elements. Its original contribution consists in an efficient strategy that combines 
 a Galerkin discretization in space via the symbol method and BICSTAB along with a circular preconditioner to efficiently solve 
 the associated linear system with non-symmetric and densely populated BTTB matrix involved in the implicit time step algorithm. Additionally, we present a practical application of this strategy to  specific \textit{crack} contracts. The paper provides numerical experiments that validate the good performance of this strategy  within the framework of the proposed models. 
 \par
 Levy models offer a more realistic approach to the dynamic of the underlying assets involved in the contract as they allow to capture rapid changes in their movement via random jumps. We propose two types of bivariate stochastic models. The first one is based on a combination of a  diffusion and random jumps. In turn, the jumps are divided into idiosyncratic and common ones, each with its specific intensity and length. It corresponds to market information that may impact both or each one separately. The model has been introduced in  \cite{oliv1}. The second model belongs to the subclass of time-changed Levy models, where the time index is consider as a non-decreasing process called subordinator. Its extension to a multivariate setting has been consider in \cite{LUCIANO20101937, 6c22611064594892998979f2779d599c}. In both models the dependence between assets is captured in a tractable way by the bivariate Brownian motion and the common jump length in the first case and an additional common subordinator in the second.
 \par
For exponential Levy models or a jump-diffusion models, the value of an European option can be obtained as the classical solution of the related Kolmogorov equation where the corresponding operator is of
partial integro-differential type. For this notion of solution, Fast Fourier Transform (FFT) based methods have been proposed in \cite{jaimungal2009stepping,lord2008fast} which are quite easy to implement, flexible and generic although they do not handle boundary conditions efficiently and are only applicable to linear problems with non space-dependent coefficients. 
Other possible notions  of solution are viscosity solutions
\cite{tankov2011pricing,alvarez1996viscosity,barles1997backward,pham1998optimal} to be approximated using finite
difference methods and variational solutions
\cite{gass2018flexible,reich2010kolmogorov,matache2004fast}
to be approximated using variational methods. In
\cite{eberlein2014variational,glau2016feynman}, important
theoretical advances are presented to establish the
connection between variational solutions and the corresponding
expectations that European option prices represent. 
The resulting relation is a Feynman-Kac-type representation
of the variational solution as a conditional expectation
providing a theoretical framework that  covers  a wide range
of Levy models. With these results, the use of the Galerkin method to approximate variational solutions acquires special relevance. 
\par
Important features of the Galerkin method are its solid theoretical background an its flexibility to choose approximation spaces and basis to address the particularities of the problem to be solved. For example,
specific finite element spaces can be constructed to handle
complicated domain geometries, non trivial boundary conditions
or to capture different degrees smoothness of the solution.
Even so, when applying the Galerkin method to PIDEs some
difficulties appear. The matrix associated to the 
operator is full, which involves two fundamental questions: \textit{how to construct it?} and (if implicit methods for time evolution are to be considered)  \textit{how to solve the
companion linear system ?}. High order algorithms for time evolution are highly
desirable since computational effort is substantially 
lower. By far the favorite approach in pricing has been the use of $\theta$-schemes. They are  easy to implement and offer flexibility to obtain  implicit algorithms allowing for larger stability regions and thus larger time steps. 
In particular the choice $\theta=1/2$ leads to the Crank-Nicholson scheme which is second order accurate and
unconditionally stable. Some wavelets-Galerkin methods has been proposed to avoid dealing with full matrices \cite{von2003wavelet,matache2005wavelet,schwab2007computational}. By using this methods, it is possible to construct a
compressed \textit{sparse}  matrix to substitute the
original full one. The compressed matrix is also
structured and only a few of its entries need to be 
computed leading to tractable linear systems. 

Recently, Gass and  Glau presented in 
\cite{gass2018flexible} a flexible Galerkin
scheme for option pricing in multidimensional Levy models,
the authors developed an approach to set up a FEM solver for option prices in Levy models that they call it the \textit{symbol method}. They show a way to implement a 
Galerkin discretization of the integro-differential operator whenever its symbol is known and the approximated solution is approximated as the linear combination  of the smooth mother basis function's displacements. By taking this approach the stiffness matrix in one dimension is Toeplitz and only its first line and first column are needed to determine it. There exist direct methods to solve a Toeplitz system in $O(N\log^2N)$ operations \cite{ammar1988superfast,bitmead1980asymptotically,bitmead1980asymptotically,de1987new,delsarte1985generalization,freund1994look} and iterative methods to solve it in $O(N\log(N))$  operations  \cite{strang1986proposal,olkin1986linear}. 
Implementing the symbol method in two dimensions lead us 
to a BTTB stiffness matrix, an introduction to iterative solvers for  Toeplitz and BTTB symmetric systems can be found in \cite{chan2007introduction}.
 
 Inspired by the use of the
Preconditioned Conjugated gradient (PCG) method  for
symmetric Toeplitz and BTTB systems covered in
\cite{chan2007introduction} we propose the use of the
Preconditioned BICGSATB (Krilov) method introduced in
\cite{van1992bi}, which works for general (non-symmetric)
matrices, along with a circular preconditioner.  As
mentioned in \cite{gass2018flexible}, the pricing error
is sensitive to the accuracy of the stiffness matrix 
entries and the authors emphasized the necessity of
accurately compute them. Pointing in that direction, 
we discretize the terms in  the integro-differential
operator not related to the pure jump
using exact Gaussian quadrature (as it is suggested in
\cite{hilber2009numerical} for the case  of classical FEM)
and suppress their corresponding terms in the symbol. 
By doing so, some amount of error is avoided since
contribution of the diffusion and the drift to the 
stiffness matrix will be exact.
\par
The paper is organized in the following way: In Section
\ref{sec:Models} we introduce the main notations and the
models to be considered based on two subclasses of 
exponential Levy models, together with their symbols and 
their associated PIDEs.  Section \ref{sec:Galerkin} is devoted
to Galerkin method for parabolic problems, the
variational formulation and discretization of the PIDE will 
be described as well as the iterative method to solve the 
BTTB system and the $\theta$-scheme. In Section \ref{sec:Numerical} we discuss the specific implementation of the method for the two proposed models, estimation problems and the numerical pricing results for \textit{crack} contracts.

\section{Notations and models }
\label{sec:Models}

For a Levy process $(X_t)_{t \geq 0}$   let the functions $\varphi_{X_{t}}(u)$ and $\Psi_{X}(u)=\frac{1}{t} \log \varphi_{X_{t}}(u)$ be their  characteristic function 
and characteristic exponent respectively. Random elements
are defined in a filtered space $(\Omega, \mathcal{A}, P,
(\mathcal{F}_t))_{t \geq 0})$ completed in the usual way.
For  row vectors $a$ and $b$ the product $a b$ refers to
its component-wise product while scalar product is 
denoted $a^Tb$, where  $a^T$ is its transpose. The 
symbol $C^{1,2}(\mathbb{R}^2)$  denote respectively the
set of functions with first and second order continuous
derivatives . We define $a_+:=\max(a,0)$.
(Triplet notation) The triplet of a Levy process is
denoted by $(b, \Sigma, \nu)$ where b , $\Sigma$ are the drift and the diffusion coefficients respectively, while $\nu$ is the Levy measure

Let  $(S_t)_{t \geq 0}=(S^{(1)}_t,S^{(2)}_t)_{t \geq 0}$ be a two dimensional stochastic process whose components represent the prices of two underlying assets, e.g. a commodity and the raw material to produce it at time $t$, and let
$(Y_t)_{t \geq 0}=(Y_t^{(1)},Y_t^{(2)})_{t \geq 0}$ be the corresponding log-return process. Both processes are related by
\begin{equation}
    S_t= S_0 \exp(Y_t)=\exp(x_0+Y_t)
    \label{eq:pricesexplevy}
\end{equation}
where $x_0=(x^{(1)}_0, x^{(2)}_0)$ and $x^{(j)}_0=\log S^{(j)}_0,\; j=1,2$. For simplicity we do not include a
mean-reversion dynamic which has been reported in
commodity movements, but the method above can be adapted
to this more realistic situation.

 The absence of arbitrage imposes the discounted underlying prices $(e^{-rt} S_t)_{t \geq 0}$ to be a $\mathcal{Q}$-martingale in each dimension under a given equivalent martingale measure  $\mathcal{Q}$ (EMM for short) which, for a Levy process,  translates to  restrictions on its triplet.
 
 A  European type spread call contract with maturity at time $T$ and strike price $K$ is a contingent claim that with a payoff  given by:
\begin{equation*}
    h(S_T)=\left( S^{(2)}_T- c S^{(1)}_T -K\right)_+
    \label{eq:spr}
\end{equation*}
where  the value $c>0$ is a conversion factor.
Commonly, contracts are written on  future prices on both assets.
Again, for simplicity we focus on the case of spot prices. The price of a  spread contract at time $t$ with initial price $S_t=x$ is denoted as $C(t,x)$ and is given by:
\begin{equation}
\begin{aligned}
C(t,x) &=e^{-r(T-t)}E_{\mathcal{Q}}\left[ h(S_t) \, | \, S_t = x\right]
\end{aligned}
\label{eq:spr2}
\end{equation}
Next, we consider specific exponential Levy models  that reflect finite and infinite jump activity in the log-prices.

\subsection{Jump-diffusion exponential models}

We look at a bivariate jump-diffusion process with common 
and idiosyncratic jumps. Such models have been studied in  \cite{oliv1}  by considering two  Poisson compound processes
with a dependence between random sizes in the common jumps.
To this end, we define two sequences of independent and identically distributed  two-dimensional random  vectors $(X_k)_{k \in \mathbb{N}}$ 
and $(X_{0,k})_{k \in \mathbb{N}}$. The   components  of the vectors in the first  sequence  are independent among them, with identical cumulative distribution function (c.d.f.) $F_{X}:=(F^{(1)}(x_1),F^{(2)}(x_2))$  and independent also from the components of vectors  $X_{0,k}$. On the other hand, the vectors $X_{0,k}=(X_{0,k}^{(1)},X_{0,k}^{(2)})_{k \in \mathbb{N}}$ have joint cumulative distribution function   $F_{X_{0}}(x)$ for every $k \in \mathbb{N}$.

Next, we introduce the bivariate homogeneous compound Poisson process  $(Z_t)_{t \geq 0}=(Z_t^{(1)},Z_t^{(2)})_{ t \geq 0}$ with components:
\begin{equation}\label{eq:jumeq}
    Z^{(j)}_t= Z^{(1,j)}_t+Z^{(0,j)}_t=\sum_{k=1}^{N^{(j)}_t} X^{(j)}_k + \sum_{k=1}^{N^{(0)}_t} X_{0,k}^{(j)},
\end{equation}
for $j=1,2$, where $(N_t)_{t \geq 0}=(  N^{(1)}_t,N^{(2)}_t)_{t \geq 0}$ is a pair of independent
Poisson processes with respective  intensities $\lambda_j>0$,
$j=1,2$, also independent of the Poisson process
$(N^{(0)}_t)_{t \geq 0}$, the latter with intensity
$\lambda_0>0$. The processes $(N^{(1)}_t)_{t \geq 0}$ and
$(N^{(2)}_t)_{t \geq 0}$  are the number of
common and idiosyncratic jumps of the underlying asset 
prices on the interval $[0,t]$, with corresponding jump 
sizes $X_{0,k}^{(j)}$ and $X^{(j)}_k$. We assume the 
existence of moments up to order two of the later.
Furthermore, we set the log-prices as jump-diffusion 
processes given by a combination of a Brownian motion and 
a compound Poisson process as:
\begin{equation}\label{eq:jdv}
  Y_t=  B_t +  Z_t 
\end{equation}
where $B_t=(B^{(1)}_t, B^{(2)}_t)$ is a pair of   Brownian motions with covariance matrix $\Sigma_B$ with
\begin{equation*}
\Sigma_B= 
\begin{pmatrix}
\sigma_1^2 & \rho_B\sigma_1\sigma_2 \\
\rho_B \sigma_1\sigma_2 & \sigma_2^2
\end{pmatrix}
\end{equation*}
The process $(Y_t)_{t \geq 0}$ has triplet $(b, \Sigma_B, \nu)$ with 
\begin{equation}
    b_j=\lambda_j E\left[X^{(j)}_k\right]+\lambda_0 E\left[X_{0,k}^{(j)}\right]
\end{equation}
and a Levy jump measure $\nu$  on $\mathbb{R}^2$ given by:
\begin{equation*}
    \nu(\text{d}x)  = \lambda_1 F_{X^{(1)}}(\text{d}x_1)  \delta_0(\text{d}x_2) +\lambda_2 \delta_0(\text{d}x_1)  F_{X^{(2)}}(\text{d}x_2)  
  + \lambda_0 F_{X_0}(dx)
\end{equation*}
The characteristic function of the log-price process is  the product of the characteristic functions of processes $Z_t$ and $B_t$. Indeed, for the compound Poisson part is easily seen that
\begin{eqnarray*}
\varphi_{Z_t}(u)&=&  \prod_{j=1}^2 E_{\mathcal{Q}}
\left[E_{\mathcal{Q}} \left[\exp\left(iu_j
\sum_{k=1}^{N_t^{(j)}}X^{(j)}_k\right)\,\bigg|\,N_t^{(j)}
\right]\right]E_{\mathcal{Q}}
\left[E_{\mathcal{Q}}\left[\exp\left(iu\sum_{k=1}^{N_t^{(0)}}X_{0,k}\right)\,\bigg|\,N_t^{(0)}\right]\right]\\
&=& \prod_{j=1}^2 \exp \left(t \lambda_j\left(\varphi_{ X^{(j)}}(u_j)-1\right)\right) \exp\left(t \lambda_0 \left(\varphi_{ X^{(0)}}(u_j)-1\right)\right)\\
&=&  \exp\left(\left( \lambda_0 (\varphi_{X_0}(u)-1)+  \sum_{j=1}^2 \lambda_j (\varphi_{ X^{(j)}}(u_j)-1) \right)t \right)
\end{eqnarray*}
where $\varphi_{X_{0,j}}$ and $\varphi_{ X^{(j)}}$ are the respective characteristic functions of common and  idiosyncratic jump sizes of the j-th asset. Hence:
\begin{eqnarray}
  \varphi_{Y_t}(u) &=& \exp \left[  \left( i u b + \frac{1}{2}u \Sigma_B u^T +  \sum_{j=1}^2 \lambda_j (\varphi_{ X^{(j)}}(u_j)-1) + \lambda_0 (\varphi_{X_0}(u)-1)\right) t \right]
  \label{eq:chfcpp}
\end{eqnarray}
Under $\mathcal{Q}$, the characteristic exponent of the log-prices verifies $\Psi^{\mathcal{Q}}_{Y^{(j)}}(-i)=r$. Therefore, the drift changes to
\begin{equation*}
\begin{aligned}
b^{\mathcal{Q}}_j &= r-\frac{1}{2}\sigma^2_j - \left( \lambda_j(\varphi_{X^{(j)}}(-i)-1) +
\lambda_2(\varphi_{X^{(2)}}(0)-1) +  \lambda_0 (\varphi_{X_{0,j}}(-i)-1)\right)
\end{aligned}
\end{equation*}
\begin{example} \label{Ex:1}(\textbf{Double Merton model with common and idiosyncratic jumps})

Consider a two dimensional jump-diffusion model given by equation (\ref{eq:pricesexplevy}),(\ref{eq:jumeq}) and (\ref{eq:jdv}) with  jump sizes following Gaussian distributions. More specifically, for any $k \in \mathbb{N}$ we    $X_k \sim N(\mu_J,D_J)$, where $D_J$ is a  diagonal matrix with components  $D_J(j,l)= \delta_{jl}(\sigma^{(j)}_J)^2$, $\mu_J=(\mu^{(1)}_{J}, \mu^{(2)}_{J})$ and $X_{0,k} \sim N(\mu_{0,J},\Sigma_{0,J})$, such that:
\begin{equation*}
    \Sigma_{0,J}=\begin{pmatrix}
\sigma^2_{0,1} & \rho_{J} \sigma_{0,1}\sigma_{0,2}\\
 \rho_{J} \sigma_{0,1}\sigma_{0,2} & \sigma^2_{0,2}
\end{pmatrix}
\end{equation*}
Here $\delta_{jl}$ is the usual Kronecker's number. In this case
\begin{eqnarray*}
\varphi_{X_0}(u) &=& \exp(-\frac{1}{2}u \Sigma_{0,J} u^T+ i  \mu_{0,J} u^T) \\
\varphi_{ X^{(j)}}(u_j) &=& \exp(-\frac{1}{2}  (\sigma_J^{(j)})^2 u^2_j + i\mu^{(j)}_J u_j) 
\end{eqnarray*}
By equation \eqref{eq:chfcpp} the characteristic exponent of the log-prices is
\begin{eqnarray} \label{eq:chfcppm}
  \Psi^\mathcal{Q}_{Y}(u) &=& i u b^{\mathcal{Q}} + \frac{1}{2}u \Sigma_B u^T +  \sum_{j=1}^2 \lambda_j (\varphi_{ X^{(j)}}(u_j)-1) + \lambda_0 (\varphi_{X_0}(u)-1)
 \end{eqnarray}
 where the drift  across each dimension in the risk-neutral measure takes the form
\begin{equation*}
  b^{\mathcal{Q}}_j= r- \frac{1}{2}\sigma^2_j-\lambda_j(\exp(\mu^{(j)}_J+ \frac{1}{2} (\sigma^{(j)}_J)^2)-1) - \lambda_0(\exp(\mu^{(j)}_{0,J} + \frac{1}{2} (\sigma^2_{0,j})-1)
\end{equation*}
\end{example}
\subsection{Time-changed exponential Levy models}
We consider now  the class of exponential bivariate time-changed Levy models, with the random time described  according to  a two dimensional subordinator of one-factor type. See    \cite{LUCIANO20101937, 6c22611064594892998979f2779d599c} for a general formulation of one-factor multidimensional time-changed processes. The dependence between both underlying assets is given through  the common  time-changed subordinator. In this context, we define the log-price process as:
 \begin{equation}\label{eq:tchan}
  Y_t= \mu R_t+\sigma_r B_{R_t}
 \end{equation}
 with
\begin{equation*}
    R^{(j)}_t=L_t^{(0)}+d_j L_t^{(j)},\;j=1,2.
    \label{eq:subor}
\end{equation*}
where  $(L_t)_{t \geq 0}=( L_t^{(1)}, L_t^{(2)})_{t \geq 0}$ is a vector of  subordinator Levy models starting at zero, with independent components and finite second moments.  As in the previous subsection the process  $(B_{t})_{t \geq 0}$ is a bivariate standard Brownian motion but with independent components. The pairs $\mu$, $\sigma_r=(\sigma_{1,r},\sigma_{2,r})$ and $d$ are parameters of the model. 

Next, we provide an expression for the characteristic function of the log-price process  below. It is a simple extension of well-known  results in the one-dimensional case.
 \begin{proposition}\label{chfcp}
Let the process $(Y_t)_{t \geq 0}$ be given by equations (\ref{eq:pricesexplevy}) and (\ref{eq:tchan}). For $t \geq 0$ and $u \in \mathbb{R}^2$ we have:
\begin{align}
 \varphi_{Y_t}(u) &= \varphi_{L^{(0)}_t}\bigg(u \mu +\frac{1}{2}i \sum_{j=1}^2 \sigma^2_{j,r} u^2_j\bigg)\varphi_{L^{(1)}_t}\left(u_1 \mu_1 d_1+\frac{1}{2}i \sigma^2_{1,r} u^2_1 d_1 \right) 
 \varphi_{L^{(2)}_t}\left(u_2 \mu_2 d_2+\frac{1}{2}i \sigma^2_{2,r} u^2_2 d_2 \right) \label{eq:chftch1} \\ 
 \Psi_{Y}(u) &= \Psi_{L^{(0)}}\bigg(u \mu +\frac{1}{2}i \sum_{j=1}^2 \sigma^2_{j,r} u^2_j\bigg)+\Psi_{L^{(1)}}\left(u_1 \mu_1 d_1+\frac{1}{2}i \sigma^2_{1,r} u^2_1 d_1 \right)+ \Psi_{L^{(2)}}\bigg(u_2 \mu_2 d_2+\frac{1}{2}i \sigma^2_{2,r} u^2_2 d_2 \bigg) \label{eq:chftch2}
\end{align}
\end{proposition}
\begin{proof}
Let the quantity $L^{(0)}_t$ represent either the one-dimensional process or the pair $(L^{(0)}_t, L^{(0)}_t)$. For $u \in \mathbb{R}^2$
\begin{equation}
\begin{aligned}
 \varphi_{Y_t}(u) &= E_{\mathcal{Q}} \left[\exp\left(i u Y_t\right)\right]  = E_{\mathcal{Q}} \left[\exp(i u (\mu R_t+ \sigma_r B_{R_t}))\right]\\ 
 &= E_{\mathcal{Q}}\left[E_{\mathcal{Q}} \left[\exp(i u (\mu R_t+ \sigma_r B_{R_t}))\,\big|\,L^{(0)}_t, L_t\right]\right]\\ 
 &= E_{\mathcal{Q}}\left[E_{\mathcal{Q}} \left[\exp(i u (\mu  (L_t^{(0)}+d L_t)+ \sigma_r B_{L_t^{(0)}+d L_t})\,\big|\,L^{(0)}_t, L_t\right]\right]\\ 
 &= E_{\mathcal{Q}}\left[ \exp(i u (\mu  (L_t^{(0)}+d L_t))) E_{\mathcal{Q}} \left[\exp( i\sigma_r u B_{L_t^{(0)}+d L_t})\,\big|\,L^{(0)}_t, L_t)\right]\right] 
 \end{aligned}
 \label{eq:chfyt}
\end{equation}
But:
\begin{eqnarray*}
 &&E_{\mathcal{Q}} \left[ \exp\left(i \sigma_r u B_{L_t^{(0)}+d L_t}\right) \big|\, L^{(0)}_t, L_t\right]= \exp\left(-\frac{1}{2}\sum_{j=1}^2 \sigma^2_{j,r} u^2_j \left(L_t^{(0)}+d_j L^{(j)}_t\right)\right)\\
  &=& \exp\left(-\frac{1}{2}\left(\sum_{j=1}^2 \sigma^2_{j,r} u^2_j L_t^{(0)}+\sum_{j=1}^2 \sigma^2_{j,r} u^2_j d_j L^{(j)}_t\right)\right)
 \end{eqnarray*}
After substituting the expression above into  equation (\ref{eq:chfyt}), from the independence of the subordinator processes, we have
\begin{equation*}
\begin{aligned}
 \varphi_{Y_t}(u) &=  E_{\mathcal{Q}}\left[ \exp\left(i u \mu  (L_t^{(0)}+d L_t)\right) \exp\left(-\frac{1}{2}\left(\sum_{j=1}^2 \sigma^2_{j,r} u^2_j L_t^{(0)}+\sum_{j=1}^2 \sigma^2_{j,r} u^2_j d_j L^{(j)}_t\right)\right)\right]\\
 &=  E_{\mathcal{Q}}\left[ \exp\left(i u \mu L_t^{(0)}-\frac{1}{2}\sum_{j=1}^2 \sigma^2_{j,r} u^2_j L_t^{(0)}\right) \exp\left(-\frac{1}{2}\sum_{j=1}^2 \sigma^2_{j,r} u^2_j d_j L^{(j)}_t+i u \mu d L_t\right)\right]\\
 &=  E_{\mathcal{Q}}\left[ \exp\left(\left(i u \mu -\frac{1}{2} \sum_{j=1}^2 \sigma^2_{j,2} u^2_j\right) L_t^{(0)}\right)\right] E_{\mathcal{Q}}\left[\exp\left(\left(-\frac{1}{2}\sigma^2_{1,r} u^2_1 d_1 +i u_1 \mu_1 d_1\right) L^{(1)}_t\right)\right]\\
 &= E_{\mathcal{Q}}\left[\exp\left(\left(-\frac{1}{2}\sigma^2_{2,r} u^2_2 d_2 +i u_2 \mu_2 d_2\right) L^{(2)}_t\right)\right]
 \label{eq:Symbol}
 \end{aligned}
 \end{equation*}
which immediately leads to equations (\ref{eq:chftch1}) and (\ref{eq:chftch2}).
\end{proof}
Moments of the log-return process $(\Delta Y_t)_{t \geq 0}$, needed for the parameter estimation in section 4, can be obtained by conditioning on the subordinator process. The result is stated below.
\begin{proposition}
 Let the process $(Y_t)_{t \geq 0}$ be given by equations (\ref{eq:pricesexplevy})-(\ref{eq:tchan}). For any $t \geq 0$, $k=1,2,3,4$  and $j=1,2$ we have
 \begin{eqnarray}\label{eq:1mom}
   E\left[\Delta Y^{(j)}_t\right] &=& \mu_j E_{j,1} \\ \label{eq:2mom}
    E\left[(\Delta Y^{(j)}_t)^2\right] &=& \mu^2_j E_{j,2}+ \sigma^2_{r,j} E_{j,1}\\ \label{eq:3mom}
    E\left[(\Delta Y^{(j)}_t)^3\right] &=& \mu^3_j E_{j,3}+ 3 \mu_j \sigma^2_{r,j}E_{j,2}\\ \label{eq:4mom} 
     E\left[(\Delta Y^{(j)}_t)^4\right] &=& \mu^4_j E_{j,4}+6 \mu^2_j \sigma^2_{r,j}E_{j,3} + \sigma^4_{r,j} E_{j,2} \\
     \label{eq:12mom}
    E\left[\Delta Y^{(1)}_t \Delta Y^{(2)}_t\right] &=& \mu_1  \mu_2 E\left[(\Delta L^{(0)}_t)^2\right]
 \end{eqnarray}
 where
 \begin{equation*}
     E_{j,k}=E\left[(\Delta L^{(0)}_t)^k\right]+d^k_j E\left[(\Delta L^{(j)}_t)^k\right]. 
 \end{equation*}
\end{proposition}
\begin{proof}
We have the relations $\Delta Y^{(j)}_t:= Y^{(j)}_{t+\Delta t}-Y^{(j)}_t=\mu_j\Delta R^{(j)}_t+ \sigma_{r,j} \Delta B^{(j)}_{R^{(j)}_t}$ and $\Delta B^{(j)}_{R^{(j)}_t}:= B^{(j)}_{R^{(j)}_t+\Delta t}-B^{(j)}_{R^{(j)}_t}$. By the independence of the subordinators  taking conditional expectations on both sides we have that $E_{j,k}=E[(\Delta R^{(j)}_t)^k]$. Furthermore,
\begin{eqnarray*}
E\left[\left(\Delta Y^{(j)}_t\right)^2\right] &=& \mu^2_j E\left[(\Delta R^{(j)}_t)^2\right]+ 2 \mu_j \sigma_{r,j} E\left[\Delta R^{(j)}_t \Delta B^{(j)}_{R^{(j)}_t}\right]+ \sigma^2_{r,j}E\left[(\Delta B^{(j)}_{R^{(j)}_t})^2\right].
\end{eqnarray*}
But
\begin{eqnarray*}
E\left[\Delta R^{(j)}_t \Delta B^{(j)}_{R^{(j)}_t}\right] &=& E\left[ E\left[\Delta R^{(j)}_t \Delta B^{(j)}_{R^{(j)}_t}\Big|\Delta R^{(j)}_t\right]\right]=E\left[ \Delta R^{(j)}_t E\left[\Delta B^{(j)}_{R^{(j)}_t}\Big|\Delta R^{(j)}_t\right]\right]=0, \\
E\left[(\Delta B^{(j)}_{R^{(j)}_t})^2\right] &=& E\left[ E\left[(\Delta B^{(j)}_{R^{(j)}_t})^2\Big|\Delta R^{(j)}_t)\right]\right]
=E\left[ \Delta R^{(j)}_t\right].
\end{eqnarray*}
Hence we have (\ref{eq:2mom}). For the third moment we have
\begin{eqnarray*}
E\left[(\Delta Y^{(j)}_t)^3\right] &=& \mu^3_j E\left[(\Delta R^{(j)}_t)^3\right]+ 3 \mu^2_j \sigma_{r,j} E\left[(\Delta R^{(j)}_t)^2 \Delta B^{(j)}_{R^{(j)}_t}\right]\\
&+& 3 \mu_j \sigma^2_{r,j}E\left[\Delta R^{(j)}_t(\Delta B^{(j)}_{R^{(j)}_t})^2\right]+ \sigma^3_{r,j}E\left[(\Delta B^{(j)}_{R^{(j)}_t})^3\right].
\end{eqnarray*}
Taking into account that
\begin{eqnarray*}
E\left[(\Delta R^{(j)}_t)^2 \Delta B^{(j)}_{R^{(j)}_t}\right] &=& E\left[ E\left[(\Delta R^{(j)}_t)^2 \Delta B^{(j)}_{R^{(j)}_t}\Big|\Delta R^{(j)}_t\right]\right]=E\left[ (\Delta R^{(j)}_t)^2 E\left[\Delta B^{(j)}_{R^{(j)}_t}\Big|\Delta R^{(j)}_t\right]\right]=0 \\
E\left[\Delta R^{(j)}_t (\Delta B^{(j)}_{R^{(j)}_t})^2\right] &=& E\left[ E\left[\Delta R^{(j)}_t (\Delta B^{(j)}_{R^{(j)}_t})^2\Big|\Delta R^{(j)}_t\right]\right]=E\left[ \Delta R^{(j)}_t E\left[(\Delta B^{(j)}_{R^{(j)}_t})^2\Big|\Delta R^{(j)}_t\right]\right] \\
&=& E\left[(\Delta R^{(j)}_t)^2\right]=E\left[(\Delta L^{(0)}_t)^2\right]+d^2_j E\left[(\Delta L^{(j)}_t)^2\right]\\
E\left[(\Delta B^{(j)}_{R^{(j)}_t})^3\right] &=& E\left[ E\left[(\Delta B^{(j)}_{R^{(j)}_t})^3\Big|\Delta R^{(j)}_t\right]\right]
=0
\end{eqnarray*}
from where we get (\ref{eq:3mom}).\\
Similarly, the forth moment is:
\begin{eqnarray*}
E\left[(\Delta Y^{(j)}_t)^4\right] &=& \mu^4_j E\left[(\Delta R^{(j)}_t)^4\right]+ 4 \mu^3 _j \sigma_{r,j} E\left[(\Delta R^{(j)}_t)^3 \Delta B^{(j)}_{R^{(j)}_t}\right]\\
&+& 6 \mu^2_j \sigma^2_{r,j}E\left[(\Delta R^{(j)}_t)^2(\Delta B^{(j)}_{R^{(j)}_t})^2\right]+ 4 \mu_j \sigma^3_{r,j}E\left[\Delta R^{(j)}_t (\Delta B^{(j)}_{R^{(j)}_t})^3\right]\\
&+& \sigma^4_{r,j}E\left[(\Delta B^{(j)}_{R^{(j)}_t})^4\right]
\end{eqnarray*}
where:
\begin{eqnarray*}
E\left[(\Delta R^{(j)}_t)^3 \Delta B^{(j)}_{R^{(j)}_t}\right] &=& E\left[ E\left[(\Delta R^{(j)}_t)^3 \Delta B^{(j)}_{R^{(j)}_t}\Big|\Delta R^{(j)}_t\right]\right]=E\left[ (\Delta R^{(j)}_t)^3 E\left[\Delta B^{(j)}_{R^{(j)}_t}\Big|\Delta R^{(j)}_t\right]\right]=0 \\
E\left[(\Delta R^{(j)}_t)^2 (\Delta B^{(j)}_{R^{(j)}_t})^2\right] &=& E\left[ E\left[\Delta R^{(j)}_t)^2 (\Delta B^{(j)}_{R^{(j)}_t})^2\Big|\Delta R^{(j)}_t\right]\right]=E\left[ (\Delta R^{(j)}_t)^2 E\left[(\Delta B^{(j)}_{R^{(j)}_t})^2\Big|\Delta R^{(j)}_t\right]\right]\\
&=&E\left[ (\Delta R^{(j)}_t)^3\right]=E\left[(\Delta L^{(0)}_t)^3\right]+d^3_j E\left[(\Delta L^{(j)}_t)^3\right]\\
E\left[\Delta R^{(j)}_t (\Delta B^{(j)}_{R^{(j)}_t})^3\right] &=& E\left[\Delta R^{(j)}_t E\left[(\Delta B^{(j)}_{R^{(j)}_t})^3\Big|\Delta R^{(j)}_t\right]\right]=0\\
E\left[(\Delta B^{(j)}_{R^{(j)}_t})^4\right] &=& 3 E\left[(\Delta R^{(j)}_t))^2\right]=3 \left(E\left[(\Delta L^{(0)}_t)^2\right]+d^2_j E\left[\Delta L^{(j)}_t\right]^2\right)
\end{eqnarray*}
Gathering all term we get (\ref{eq:4mom}). Finally, for the mixed moment
\begin{eqnarray*}
 E\left[\Delta Y^{(1)}_t  \Delta Y^{(2)}_t\right] &=& E\left[(\mu_1\Delta R^{(1)}_t+ \sigma_{r,1} \Delta B^{(1)}_{R^{(1)}_t})(\mu_2\Delta R^{(2)}_t+ \sigma_{r,2} \Delta B^{(2)}_{R^{(2)}_t})\right]\\
 &=& \mu_1 \mu_2 E\left[\Delta R^{(1)}_t \Delta R^{(2)}_t\right]+  \mu_1 \sigma_{r,2} E\left[\Delta R^{(1)}_t \Delta B^{(2)}_{R^{(2)}_t}\right] \\
 &+&  \mu_2 \sigma_{r,1} E\left[\Delta R^{(2)}_t \Delta B^{(1)}_{R^{(1)}_t}\right]
 +\sigma_{r,1} \sigma_{r,2} E\left[ \Delta B^{(1)}_{R^{(1)}_t}\Delta B^{(2)}_{R^{(2)}_t}\right]
\end{eqnarray*}
From the expressions:
\begin{eqnarray*}
E\left[\Delta R^{(1)}_t \Delta R^{(2)}_t\right] &=& E\left[(\Delta L^{(0)}_t)^2\right]+ d_1 d_2  E\left[\Delta L^{(1)}_t \Delta L^{(2)}_t\right]=E\left[(\Delta L^{(0)}_t)^2\right]\\
E\left[\Delta R^{(1)}_t \Delta B^{(2)}_{R^{(2)}_t}\right] &=& E\left[\Delta R^{(2)}_t \Delta B^{(1)}_{R^{(1)}_t}\right]
= E\left[ \Delta B^{(1)}_{R^{(1)}_t}\Delta B^{(2)}_{R^{(2)}_t}\right]=0 
\end{eqnarray*}
 we have (\ref{eq:12mom}).
 \end{proof}
   \begin{example} \label{Ex:2} (Two dimensional time-changed process with Gamma subordinator)\\
  Consider subordinators $(L^{(j)}_t)_{t \geq 0}, j=0.1.2$   of Gamma type. Their characteristic functions and exponent are:
  \begin{eqnarray*}
 \varphi_{L^{(l)}}(u)&=&  \left( 1-\frac{iu}{\beta_l} \right)^{-\alpha_l t}, \; \alpha_l, \beta_l > 0,\;l=0,1,2 \\
  \Psi_{L^{(l)}}(u)&=&-\alpha_l \log \left( 1-\frac{iu}{\beta_l} \right), \; \alpha_l, \beta_l > 0,\;l=0,1,2
 \end{eqnarray*}
 Notice that :
 \begin{equation*}
 E_{j,k}=\frac{\prod_{l=0}^{k-1}(\alpha_0 \Delta t+l)}{\beta^k_0}+d^k_j \frac{\prod_{l=0}^{k-1}(\alpha_j \Delta t+l)}{\beta^k_j},\;k=1,2,3,4; j=1,2
\end{equation*}
 Hey proposition 1:
 \begin{equation*}
    \begin{aligned}
 \Psi_{Y}(u) &=& - \sum_{l=0}^2  \alpha_l \log \left( 1-\frac{iu}{\beta_l} \right)
 \end{aligned}
\end{equation*}
The marginal characteristic exponents  of $(Y_t)_{t \geq 0}$ are obtained also from equat(\ref{eq:chftch2})  as:
\begin{eqnarray*}
   \Psi_{Y^{(1)}_t}(u_1)&=& \Psi_{Y_t}(u_1,0)= -\alpha_0 \log \left( 1-\frac{i\mu_1}{\beta_0}u_1+\frac{\sigma^2_{1,r} u^2_1}{2 \beta_0} \right)\\
   &-& \alpha_1 \log \left( 1-\frac{i\mu_1 d_1}{\beta_1}u_1+\frac{\sigma^2_{1,r} u^2_1}{2  \beta_1} \right) \\
   \Psi_{Y^{(2)}}(u_2)&=&\Psi_{Y_t}(0,u_2)= -\alpha_0 \log \left( 1-\frac{i\mu_2}{\beta_0}u_2+\frac{\sigma^2_{2,r} u^2_2}{2 \beta_0} \right)\\
   &-& \alpha_2 \log \left( 1-\frac{i\mu_2 d_2}{\beta_1}u_2+\frac{\sigma^2_{2,r} u^2_2}{2 \beta_2} \right)
\end{eqnarray*}
Hence, the martingale conditions $\Psi^{(j)}(-i)=r$ translates into the equations:
\begin{eqnarray*} 
    \label{eq:FreeRisk}
    && -\alpha_0 \log \left( 1-\frac{\mu_j}{\beta_0}-\frac{\sigma^2_{j,r}}{2 \beta_0} \right)-\alpha_j \log \left( 1-\frac{\mu_j d_j}{\beta_j}-\frac{\sigma^2_{j,r} }{2 \beta_j} \right)=r,\;\;\ j=1,2
\end{eqnarray*}
\begin{figure}[ht]\label{fig:trajWTI}
    \begin{minipage}[b]{0.4\textwidth} 
    \includegraphics[trim={4cm 3cm 6cm 2cm},clip,scale=.3]{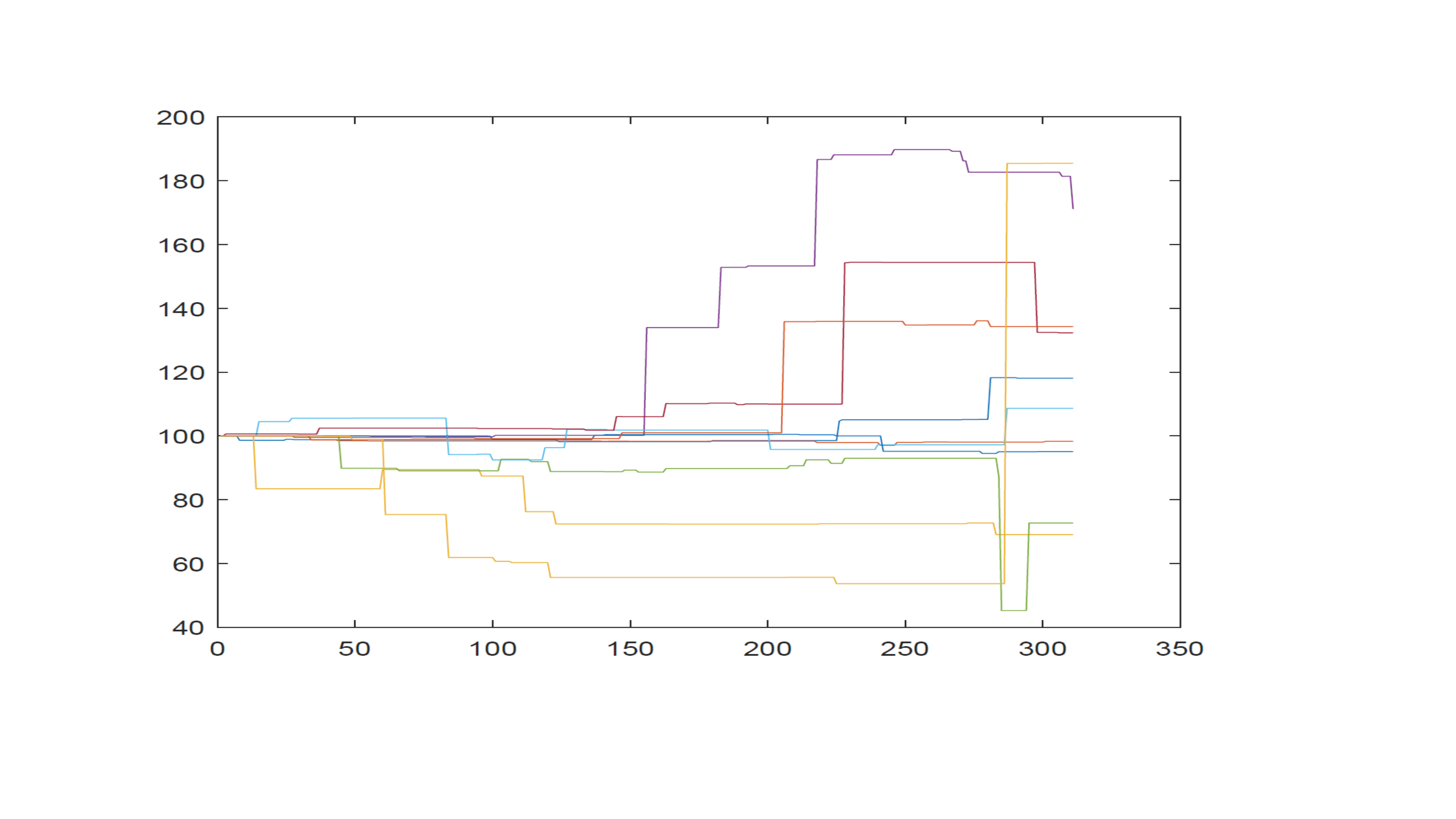}
    \caption{Simulated trajectories of WTI prices accordingly to a time-changed process under the Risk Neutral measure. Selection of parameters are discussed in section 4 }
    \end{minipage}
    \hspace{2cm}
    \begin{minipage}[b]{0.4\textwidth}
    \hspace{-.5cm}
    \includegraphics[trim={4cm 2cm 4cm 2cm},clip,scale=.3]{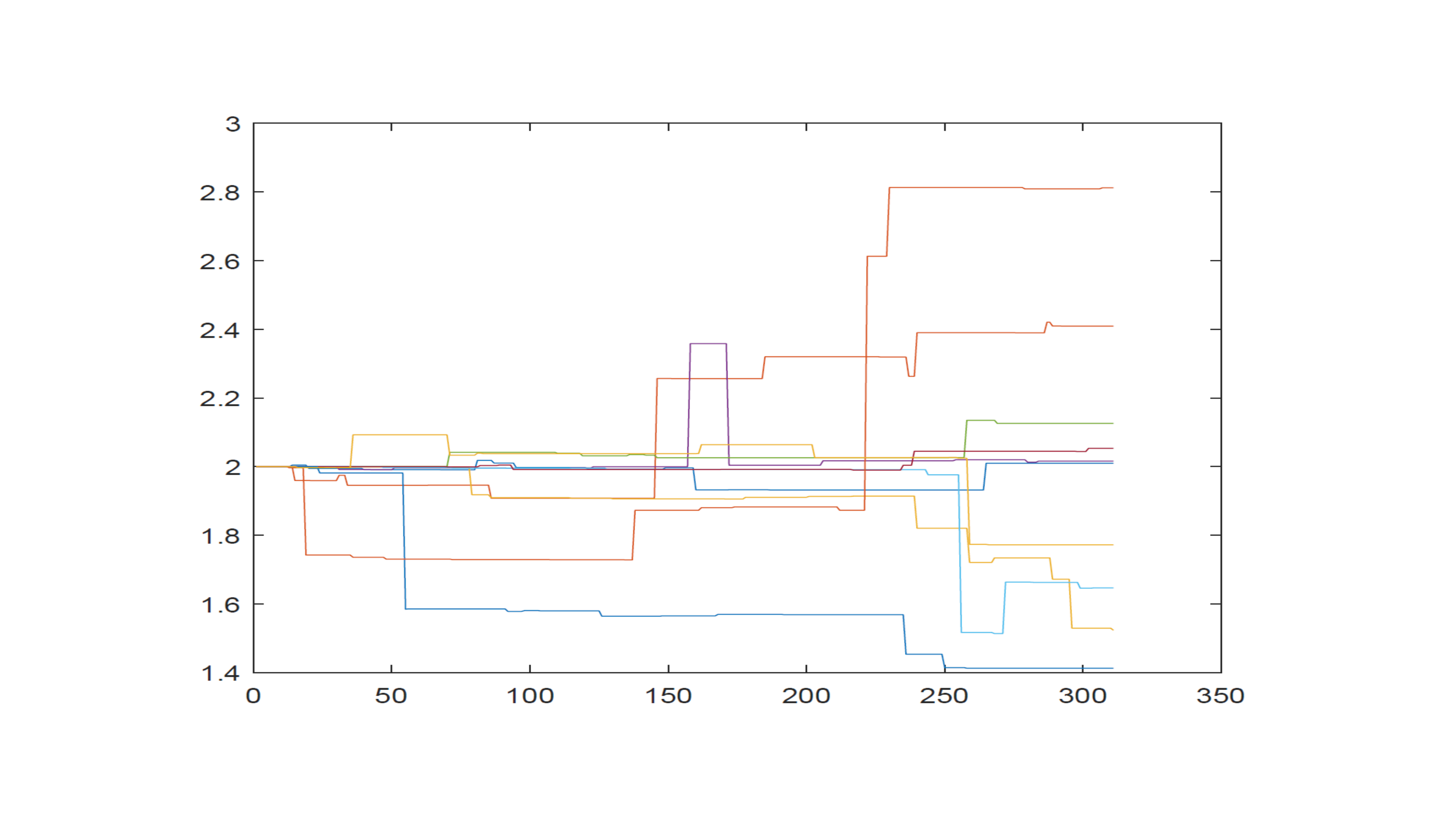}
    \caption{Simulated trajectories of RBOB future prices accordingly to a time-changed process under the Risk Neutral measure. Selection of parameters are discussed in section 4 }
    \end{minipage}
\end{figure}
\end{example}

\subsection{Associated PIDE}  
 
Consider the interval $I=[0,T] \subset\mathbb{R}$, a terminal condition $h:\Re^2\rightarrow \mathbb{R}$ and a function $C:I \times \Re^2\rightarrow\mathbb{R}$.  The Kolmogorov equation is of the form
\begin{equation}
\begin{aligned}
     \partial_t C+ \mathcal{A} C &= 0,\\
     C(T) &= h,
\end{aligned}
\label{eq:Kolmogorov}
\end{equation}
where  $\mathcal{A}$ is the Kolmogorov operator of a Levy model and is given by
\begin{equation}
   \begin{aligned}
  \mathcal{A} C(t,x) = &\sum_{j,k=1}^2 \frac{\sigma^2_{jk}}{2}\frac{\partial^2 C}{\partial x_j \partial x_k}(t,x)+ \sum_{j=1}^2 b_j \frac{\partial C}{\partial x_j}(t,x)  \\
  &+\int_{\mathbb{R}^2} \bigg[C(t,y+x)-C(t,x)- \sum_{j=1}^2 y_j \frac{\partial C}{\partial x_j}(t,x) 1_{|y_j| <1}\bigg]\nu(dy).
\end{aligned} 
\end{equation}
for all $t\in I$ and $x\in\Re^2$, with symbol
\begin{equation*}
    A(\xi):=\frac{1}{2}\langle\xi,\sigma\xi\rangle + i\langle\xi,b \rangle - \int_\mathbb{R}^2 \left( \text{e}^{-i\langle\xi,y \rangle}-1+i\langle\xi,\textbf{h}(y)\rangle\right)\nu(\text{d}y) = -\Psi(-i\xi).
    \label{eq:Symbol_Levy}
\end{equation*}
If we change variable $\tau=T-t$ and define
\begin{equation*}
  u(\tau,x)=\exp(r \tau) C(T-\tau, e^{x^{(2)}_0+x_2}-c e^{x^{(1)}_0+x_1})
\end{equation*}
we have that $u$ satisfies equation.
\begin{equation}
   \begin{aligned}
 \partial_t u+ \mathcal{A} u + r u(\tau,x)&= 0\\
 u(0)&=h, \\
\end{aligned} 
\label{eq:Kolmogorov_Initial}
\end{equation}

\section{Variational formulation, Galerkin method and BTTB systems} \label{sec:Galerkin}

The Galerkin method is based in the variational formulation of problem \eqref{eq:Kolmogorov_Initial} and our choice for an implicit $\theta$-scheme requires the resolution of a BTTB system with non-symmetric and densely populated matrix at each time step. 

\subsection{Variational formulation}
Assuming that \eqref{eq:Kolmogorov_Initial} has a classical
solution $u \in C^{1,2}(\Re^2)$ (sufficiently regular to
apply Ito's formula) while satisfying an appropriate
integrability condition on $h$, there is a Feynman-Kac type
representation for $u(\tau)$ so that $C(t)$ has the stochastic representation \eqref{eq:spr2}. But, since we are considering 
a variational formulation of \eqref{eq:Kolmogorov_Initial}, 
such hypothesis turns out to be very restrictive. We are 
looking for less smooth solutions.
Let the Gelfand triplet $(V,H,V^*)$ be composed by separable
Hilbert spaces $H$, $V$ and the dual space $V^*$  such that
there exist a continuous embedding from $V$ into $H$. Denote
$\left(\cdot,\cdot\right)_H$ the inner product of $H$ and
$\langle\cdot,\cdot\rangle_{V^*\times V}$ the duality pairing.
Let $L^2(I,H)$ be the space of weekly measurable functions $u: I \rightarrow H$ with $\int_0^T\| u(t)\|_H^2 dt \leq \infty$. Weak solutions of \eqref{eq:Kolmogorov} will belong to the Sobolev space
\begin{equation}
W^1(I,V,H) := \lbrace u\in L^2(I,V) | \partial_tu \in L^2(I,V^*)\rbrace
\end{equation}
The variational formulation of problem \eqref{eq:Kolmogorov_Initial} results now in: \textit{To find}  $u \in W^1(I,V,H)$ \textit{such that}
\begin{equation}
\begin{aligned}
\left( \partial_t u(t),v \right)_H + a(u(t),v) + r\langle u(t),v\rangle_H &= 0  && \forall v\in V \\
u(t) &\rightarrow h 
\end{aligned}
\label{eq:weakform}
\end{equation}
where $h\in H$, convergence in the second line above is in the norm of $H$ and
\begin{equation}
a(u(t),v) = \left(  \mathcal{A}u(t),v\right)_H,
\end{equation}
We adopt the theoretical framework in \cite{glau2016feynman,eberlein2014variational,gass2018flexible,glau2016classification} and consider the Sobolev–Slobodeckii spaces $V = H^\alpha_\eta(\mathbb{R}^2)$ and $H=L^2_\eta(\mathbb{R}^2)$. For concrete definitions of weighted Sobolev–Slobodeckii spaces see \cite{eberlein2014variational,glau2016feynman}.  Two matters are to be considered, one related to \textit{existence} and \textit{uniqueness} of the solution of \eqref{eq:weakform}  and the other related to the equivalence
between such a solution and \eqref{eq:spr2}. Both problems has been addressed in \cite{eberlein2014variational} and \cite{glau2016feynman}. Well posedness  of problem
\eqref{eq:weakform} is guarantee by \cite[Theorem 5.3 pag. 15]{eberlein2014variational} if hypotheses (A1), (A2) and (A3) are to be satisfied by the symbol $A$ whenever the payoff $h\in L^2_\eta(\mathbb{R}^2)$. 
\begin{itemize}
    \item[(A1)] \textit{Assume that}
    \begin{equation*}
        \int_{|x|>1}\text{e}^{-\langle\eta',x \rangle}\nu(\text{d}x)<\infty \quad\quad\forall\eta'\in R_\eta
    \end{equation*}
    \item[(A2)] \textit{There exist a constant $C_1>0$ with}
\begin{equation*}
    \mid A(u) \mid\leq C_1\left( 1 + |u|\right)^\alpha \quad\quad \forall u\in U_{-\eta} \quad\quad\quad\quad\quad\quad (Continuity\, condition)
\end{equation*}
    \item[(A3)] There exist constants $C2>0$ and $C3 \leq 0$ uniformly in time, such that for a certain $0\leq \beta <\alpha$
    \begin{equation*}
        \mathbb{R}\left(A(u)\right)\geq C2\left(1+|u| \right)^\alpha -C_3\left(1+|u| \right)^\beta \quad\quad \forall u\in U_{-\eta} \quad\quad (G\mathring{a}rding \,condition)
    \end{equation*}
\end{itemize}
For European spread options the choice  
$ \eta\in(0,a-2)\times(-\infty,-a)$ for any $a>2$ 
guaranties that  $h \in L^2_\eta(\mathbb{R}^2)\cap
L^1_\eta(\mathbb{R}^2)$, as it is proved in \ref{A:parameter}.  
The case of the Double Merton model with common and
idiosyncratic jumps in Example \ref{Ex:1} we have that
\begin{align*}
 \int_0^T \int_{|x|\geq1}\text{e}^{\langle\eta',x \rangle}\nu(\text{d}x)\text{d}s \leq \infty & \quad\quad \forall \eta'\in\mathbb{R}^2
\end{align*}
which is a stronger condition than \textit{(A1)},
conditions \textit{(A2)} and \textit{(A3)} hold, see \ref{A:Continuity} and \ref{A:Garding} for the proof. Applying \cite[Theorem 5.3]{eberlein2014variational},  
there exist a unique solution of $u \in W^1(0,T;H_\eta^1(\mathbb{R}^2),L^2_\eta(\mathbb{R}^2))$ 
for problem \eqref{eq:weakform2}. Under the additional condition that $h\in L^1_\eta(\mathbb{R}^2)$ and as a consequence of \cite[Theorem 6.1 pag. 17]{eberlein2014variational} we have that the solution $u$ has the stochastic representation \ref{eq:spr2}.
\begin{remark}
It doesn't exist an index $\alpha$ for Example \ref{Ex:2} so we can not use the results above in this case. Even so, we applied the symbol method to this example and the numerical results in Section \ref{sec:Numerical} show convergence evidence.
\end{remark}
We separate the Kolmogorov operator in two parts, one grouping Diffusion, Drift and  and the other will be the integral operator $\mathcal{A} = \mathcal{BS} + \mathcal{I}$
\begin{equation}
   \begin{aligned}
  &\mathcal{BS}u(t,\textbf{x}) &=&\sum_{j,k=1}^2 \frac{\sigma^2_{jk}}{2}\frac{\partial^2 u}{\partial x_j \partial x_k}(t,\textbf{x})+ \sum_{j=1}^2 b_j \frac{\partial u}{\partial x_j}(t,\textbf{x}) + r\,u(t,\textbf{x}) \\
  &\mathcal{J}u(t,\textbf{x}) &=&\int_{\mathbb{R}^2} \bigg[u(t,\textbf{y}+\textbf{x})-u(t,\textbf{x})- \sum_{j=1}^2 y_j \frac{\partial u}{\partial x_j}(t,\textbf{x}) 1_{|y_j| <1}\bigg]\nu(dy).
\end{aligned} 
\end{equation}
for all $t\in I$ and $\textbf{x} \in \Re^2$. The bilinear form $a$ can now be written as 
\begin{equation}
\begin{aligned}
a(u(t),v) &= \left(\mathcal{BS}u(t),v\right)_H + \left(\mathcal{J}u(t),v\right)_H.
\end{aligned}
\end{equation}
The solution of the \eqref{eq:weakform} is defined in the whole $\mathbb{R}^2$ plane, in order to implement a numerical method we need first to localize to a bounded domain. A standard procedure  is to find suitable function $\psi$  asymptotically similar to the exact solution when $x\rightarrow \infty$ (see \cite{gass2018flexible,miglio2008finite}). By subtracting $\psi: I\times\Re^2\rightarrow\Re$ from $u$  we obtain a new equation equivalent to \eqref{eq:Kolmogorov_Initial} to be solved for 
$v = u-\psi$ 
\begin{equation}
\begin{aligned}
\partial_t v + \mathcal{BS}v + \mathcal{J}v &= f_\psi \\
v(0) &= h_{\psi}, 
\end{aligned}
\label{eq:vartrunc}
\end{equation}
where $h_{\psi} = h - \psi$  and the right hand side $f$ is obtained by 
\begin{equation}
f_\psi =  \left(\partial_t\psi + \mathcal{BS}\psi + \mathcal{J}\psi\right).
\end{equation}
And we obtain a variational formulation equivalent to \eqref{eq:weakform} by following the steps above: \textit{To find}  $v \in W^1(I;V,H)$ \textit{such that}
\begin{equation}
\begin{aligned}
( \partial_t v,\varphi )_H + (\mathcal{BS}v,\varphi)_H + (\mathcal{J}v,\varphi)_H  &= \langle f_\psi,\varphi) \rangle_{V^*\times V}  && \forall \varphi\in V \\
v(t\rightarrow 0) &= h_\psi 
\end{aligned}
\label{eq:weakform2}
\end{equation}
 The original solution of the problem is restored $u = v + \psi$. There are different choices for $\psi$, in this work the choice will be just the terminal condition. Boundary
 conditions to be imposed are at most an approximation of 
 the real value of the solution outside the domain of 
 interest, that's why a typical practice is to consider 
 a wider outer domain containing the domain of interest 
 to keep the error associated to the inaccuracies of the boundary conditions away. The wider the  outer domain is 
 the less the error propagates inside the domain of interest. The  associated truncation error can be estimated, see for example \cite[Chapter 2.3]{clift2007linear}.

\subsection{Galerkin method, $\theta$-scheme and symbol method}

We proceed as usual by truncating the unbounded range
$\Re^2$ to a bounded computational domain $\Omega$ 
and considering Hilbert spaces  $H = L^2(\Omega)$ 
and $V = H^1_0(\Omega)$ on $\Omega$ with the usual
inner product $(\cdot,\cdot)_{L^2}$. Since
$H^1_0(\Omega)$ is separable it contains a countable
basis $\{\varphi_k\}_{k\in\mathbb{N}}$ and we restrict 
our selves to a finite dimensional space 
$V^N \subset H^1_0(\Omega)$ with basis 
$\lbrace \varphi_1, \varphi_2,...,\varphi_N \rbrace$
where we consider approximated solutions of the form
$u(x) = \sum_{k=1}^N U_k(t)\varphi_k$ and an
approximation of the initial condition 
$h_N = \sum_{k=1}^N h_k\varphi_k$. The Galerkin
formulation of (\ref{eq:weakform}) now reads :
\textit{Find} $U_k(t)\in \Re^N$ \textit{such that}
\begin{align}
\sum^N_{k=1}\dot{U}_k(t)(\varphi_k,\varphi_j)_{L^2} + \sum^N_{k=1}U_k(t)(\mathcal{BS}\varphi_k,\varphi_j)_{L^2} + \sum^N_{k=1}U_k(t) (\mathcal{J}\varphi_k,\varphi_j)_{L^2} &= (f_\psi(t),\varphi_j)_{L^2} \label{eq:ODEs}\\
U_j(0) &= (h_\psi,\varphi_j)_{L^2},\label{eq:ODEs_Init}
\end{align}
for all $j=1,...,N$. There might be several possible choices for $\psi$, in present work we use the terminal condition (payoff). As $\psi$ is just an approximation of the solution outside the domain of interest, a typical practice is to consider solve the problem in the outer domain $\Omega$ containing the domain of interest to keep away the error associated to the inaccuracies of the in boundary. The wider the  outer domain is the less the propagation of the error inside the domain of interest. The error associated to boundary conditions can be predicted, 
see \cite[Chapter 2.3]{clift2007linear}. 
The  problem \eqref{eq:ODEs}-\eqref{eq:ODEs_Init} can be written as a system of ODEs
\begin{equation}
\begin{aligned}
\textbf{M}\dot{U}(t) + \textbf{A}U(t) + \textbf{J}U(t) + r\textbf{M} U(t)&= F(t)
\end{aligned}
\label{eq:EDOSyst}
\end{equation}
with $U(0)$ as in \eqref{eq:ODEs_Init}, $F(t) = [F_1(t), ...\, ,F_N(t)]  = \{(f(t),\varphi_k)_{L^2}\}_{k=1}^N$,  mass matrix \textbf{M}, stiffness matrix \textbf{A}  and \textbf{J}  the matrix associated to the pure jump operator. 
\begin{align}
\textbf{M} = (\textbf{M}_{kj}) = (\varphi_j , \varphi_k)_{L^2}, && \textbf{A} = (\textbf{A}_{kj}) = (\mathcal{BS}\varphi_j , \varphi_k)_H && \textbf{J} = (\textbf{J}_{kj}) = (\mathcal{J}\varphi_j , \varphi_k)_{L^2}
\label{eq:matrices}
\end{align}
 To solve the ODE system (\ref{eq:EDOSyst}) let us consider a uniform partition $\lbrace t_0,t_2, ...\, ,t_M \rbrace$ of $I$ with norm's partition $\Delta t = T/M$. We  approximate the coefficients $U(t)$ at the partition points, we call $U^k = U(t_k)$ to be obtained by solving the following semi-implicit $\theta$-Scheme
\begin{equation}
U^{k+1} =(\textbf{M}+\Delta t\,\theta \,(\textbf{A}+\textbf{J}+r\textbf{M}))^{-1} \left((\textbf{M} - \Delta t (1 - \theta)(\textbf{A} + \textbf{J}+r\textbf{M})) U^k  + \theta\left(F^{k}\right) + (1-\theta) F^{k+1}\right)
\label{eq:thetasheme}
\end{equation} 
 Now we specify the approximation space we use along this work. Let $\Omega$ be a square domain and a family of meshes $\mathcal{T}^h(\Omega)$ defined by identical squares where $h$ is the norm of the mesh coinciding with the length of the square's side. Let $V_N$ be the space of all continuous functions with continuous second derivatives that are piecewise third  order polynomials on each element of $\mathcal{T}^h(\Omega)$ and let $s_3$ be the Irwin-Hall cubic spline on the pattern interval $[-2,2]$.
\begin{equation*}
    s_3(x) = 
    \begin{dcases}
     \frac{1}{4}  (x+2)^2           & -2 \leq x <-1 \\
     \frac{1}{4}( 3|x|^3 - 6x^2 +4) & -1 \leq x <1 \\
      \frac{1}{4}(2-x)^3           & \,\,\, 1 \leq x \leq 2 \\
      0                 & \,\,\, \text{otherwise}
      \end{dcases}
\end{equation*}
Define the mother base $\varphi_0(\textbf{x}) = s_3(h^{-1}
x_1)  s_3(h^{-1} x_2)$ and for each  node
$\textbf{x}^{(j)}$ associated to the mesh $\mathcal{T}^h$ we define the nodal function $\varphi_j$ as the displacement of $\varphi_0$ 
 \begin{equation}
     \varphi_{j}(\textbf{x}) = \varphi_0(\textbf{x} - \textbf{x}^{(j)}) 
     \label{eq:nodalbasis}
 \end{equation}
The set of all nodal basis \eqref{eq:nodalbasis} is a basis for $V_N$. So we have set the problem in the classical finite elements framework. The first two integrals in \eqref{eq:matrices}  can be solved using Gaussian Legendre quadrature. The Matrix $\textbf{J}$ in \eqref{eq:matrices} will be approximated by the \textit{symbol method for stiffness matrices} as in \cite[Corollary 4.4]{gass2018flexible}. 
\begin{equation}
    \textbf{J}_{kj}=\left(\mathcal{J}[\varphi_j],\varphi_k\right)_{L^2} = \frac{1}{(2\pi)^2}\int_{\Re^2} J(\xi)e^{i\xi\cdot\left(\textbf{x}^{(j)} - \textbf{x}^{(k)}\right)}|\widehat{\varphi}_0(\xi)|^2\text{d}\xi,
    \label{eq:SymbolMeth}
\end{equation}
here $J$ is the symbol of the pure jump process and using \cite[Lemma 4.11]{gass2018flexible} we get
\begin{equation*}
    \widehat{\varphi}_0(\xi) = \left(\frac{3}{\xi_1^4}\left(\cos(2\xi_1) -4\cos(\xi_1)+3\right)\right)\left(\frac{3}{\xi_2^4}\left(\cos(2\xi_2) -4\cos(\xi_2)+3\right)\right),
\end{equation*}
for all $\xi = (\xi_1,\xi_2)\in\Re^2$. Integrals in \eqref{eq:SymbolMeth} are suitable to be computed using
Fast Fourier Transform so $\mathcal{O}\left( N\log(N)\right)$
operations are involved in this step. Matrix $\textbf{J}$ 
is full but is still a BTTB matrix and so are $\textbf{M}$ 
and $\textbf{A}$ in \eqref{eq:matrices} as  well as their
linear combinations.

\subsection{BTTB Systems}

The general $m$-by-$m$
block Toeplitz matrix with $n$-by-$n$ Toeplitz blocks is
defined
\begin{equation}
T_{mn} = 
\begin{bmatrix}
T_{0} & T_{-1} &   &  & T_{-n+2}&T_{-n+1}\\
T_{1} & T_{0} & \ddots & &  & T_{-n+2} \\
      & T_{1} &\ddots &   &  &   \\
      &    & \ddots &  &    & \\
T_{n-2} &   &    &  &   & T_{-1}\\
T_{n-1} & T_{n-2}   & \dots &   & T_{1}   & T_{0}.
\end{bmatrix}
\label{eq:bttbmatrix}
\end{equation}
Matrix $T_{mn}$ has a blocked Toeplitz structure where each block $T_l$, for $|l| \leq m-1$, is itself a Toeplitz matrix of order $n$.  There exist direct methods that solve BTTB systems in $\mathcal{O}(mn\log^2(mn))$ (see \cite{kumar1985fast}). Another
approach  is to use iterative methods, a 
brief description of Preconditioned Conjugated Gradient (PCG) 
based method is presented in \cite[Chapter 5]{chan2007introduction}  and two different circulat preconditioners are considered to
cluster the eigenvalues of $T_{mn}$ around $1$. Of course, the 
use of PCG is suitable when matrix $T_{mn}$ is symmetric which 
is not our case since the action of the drift breaks-down the operator symmetry. 

We use the bi-conjugated gradient stabilized method (BICGSTAB) which is suitable for non-symmetric linear systems of equations. Fortunately, the circulant preconditioner $c_{F\otimes F}$ presented in \cite{chan2007introduction} works for general BTTB matrices. Also, BTTB matrices vector product $T_{mn}b$ can be performed in $\mathcal{O}(mn\log(mn))$ operations (see \cite{chan2007introduction}). The evaluation is carried on by
first embedding each block $T_k$ of matrix $T_{mn}$ in \eqref{eq:bttbmatrix} in a circular matrix $C_k\in\Re^{n\times n}$ as follows, the block $T_k$ is defined by its first column $v_1 = T_k(:,1)$ and its first line (with zero in the first position) $v_2 = [0,T_k(1,2:\text{end})]$, and define $C_k$ by defining its first column $C_k(:,1) = [v_1,v_2]$. In a similar way we define $C_{2m,2n}$ by defining its first  block column as $C_{2m,2n}(:,1:2n) = \left[C_0,...,C_{1-m},\textbf{0}_{2n},C_1,...,C_{m-1} \right]$. Matrix $C_{2m,2n}$ can  be decomposed
\begin{equation}
    C_{2m,2n} = (I_{2m} \otimes F_{2n})^*\cdot P^T \cdot(I_{2n}\otimes F_{2m})^*\cdot (\delta(F_{2m}F_{2n} C))\cdot(I_{2m} \otimes F_{2m}) \cdot P\cdot(I_{2n}\otimes F_{2n}),
    \label{eq:circular}
\end{equation}
where $F_k$ is the Fourier matrix of dimension $k\times k$, $C$ is the matrix formed by the first column of the component wise circulant matrix in which each Toeplitz matrix entry is embedded; and $P$ is a permutation matrix that reorders columns of the identity $I_{4mn}$ in the following way
\begin{equation*}
    i \rightarrow 2n\cdot \left[(i-1)(\mod 2m)\right] + \floor*{\frac{i-1}{2m}}  + 1
\end{equation*}
being $i$ the index of the columns of $I_{4mn}$ . Matrix $P$ reorders the result of $F_{2n}C$ in a block diagonal Toeplitz matrix. After that we perform the product  $C_{2m,2n}\overline{b}$ applying the formula \eqref{eq:circular}, where $\overline{b}$ contains the entries of vector $b = [b_1,b_2,...,b_m]^T$ with $b_j\in\Re^n$ and is filled by zeros as follows
\begin{equation}
\overline{b}  = \bigg[ \Big[ \big[b_1,\textbf{0}_n\big], ... ,\big[ b_m,\textbf{0}_n \big] \Big],\textbf{0}_{2n\cdot m} \bigg] 
\label{eq:filledB}
\end{equation}
and the solution of the original product is obtaining by cropping the positions where zeros appear in \eqref{eq:filledB}. For a detailed explanation see \cite{gray2006toeplitz,chan2007introduction} 
Now we present the definition of the preconditioner $c_{F\otimes F}$ and a formula to obtain it. Consider a general blocked matrix
 \begin{equation*}
A_{mn} = 
\begin{bmatrix}
A_{1,1} & A_{1,2} & \dots & A_{1,m} \\
A_{2,1} & A_{2,2} & \dots & A_{2,m} \\
\vdots  & \ddots  &\ddots & \vdots  \\
A_{m,1} & A_{m,2}   & \dots & A_{m,m}
\end{bmatrix}
\label{eq:bttbmatrix1}
\end{equation*}
where  $A_{i,j}\in\Re^{n\times n}$. Given $2$ unitary matrices $V$ and $U$, let
\begin{equation}
    \mathcal{M}_{V\otimes U} \equiv \{ (V\otimes U)^*\Lambda_{mn} (V\otimes U) : \Lambda_{mn} \text{ is a diagonal } mn \times mn \text{ matrix}\}
\end{equation}
 the preconditioner $c_{V\otimes U}(A_{mn})$ is defined as
\begin{equation*}
    c_{V\otimes U}(A_{mn}) = \arg\min_{W_{mn}\in\mathcal{M}_{V\otimes U}} \|A_{mn} - W_{mn} \|_{\mathcal{F}}
\end{equation*}
where  $\|\cdot\|_{\mathcal{F}}$ is the Forbenius norm. If it is the case that $U = F$ and $V = F$  we have the fallowing formula
(see \cite[eq. 5.17]{chan2007introduction})
\begin{equation}
     c_{F\otimes F}(A_{mn}) = \frac{1}{mn}\sum_{j=0}^{m-1}\sum_{j=0}^{n-1} \left( \sum_{p-q\equiv j(\text{ mod } n)} \sum_{r-s\equiv k(\text{ mod } n)} (A_{pq})_{rs}\right) (Q^j\otimes Q^k)
    \label{eq:circulantCoef}
\end{equation}
where $Q$ is an $n$-by-$n$ circulant matrix given by
\begin{equation*}
Q = 
\begin{bmatrix}
0 &        &        &        &   1 \\
1 & 0      &        &        &     \\
0 & 1      & \ddots &        &     \\
\vdots & \ddots & \ddots & \ddots & \\
0 & \dots  & 0   & 1   & 0
\end{bmatrix}
\label{eq:bttbmatrix2}
\end{equation*}
It is important to notice that $c_{F\otimes F}(T_{mn})$ is a circulant matrix determined by its first column so there is no need to compute all entries in \eqref{eq:circulantCoef} which leads, of course, to a considerable reduction in computational time.  Note also that by using iterative methods we just need to perform  matrix vector products, so each iteration of the method will be done in $\mathcal{O}(nm\log(nm))$ and its overall performance will preserve that order as long as the preconditioner allows for just a few iterations. 

For a closer approach to this and others circulant preconditioners and its properties see \cite{chan2007introduction} and the precedents works \cite{chan1988optimal,chan1991circulant,tyrtyshnikov1992optimal}. Other iterative methods can be considered, e.g. Generalized Minimal Residual (GMRES) method introduced in \cite{saad1986gmres}, which is also a Krilov method; or classic non-Krilov  iterative methods such as SOR, Gauss-Seidel or Jacobi. In any case, preconditioning will play an essential roll in the method's performance.

\section{Numerical results}
\label{sec:Numerical}

By using a Galerkin approximation in space and a
$\theta$-scheme ($\theta = \frac{1}{2}$, Crank-Nicholson) for time evolution look for an approximate solution of problem \eqref{eq:vartrunc} as shown in Section \ref{sec:Galerkin}. 
By doing so we approximate the price of spread contracts 
associated to the models introduced in examples \ref{Ex:1} 
and \ref{Ex:2}.  In particular, we focus on \textit{cracks} contracts based  on the difference between the price of oil and gasoline. Specifically, we consider the prices of futures RBOB 
gasoline (reformulated blendstock for oxygenate blending) 
and West Texas Intermediate(WTI) oil in NYMEX.

The contract specifications consist on a series of strike
prices around $K=1$ US dollar, motivated by typical
\textit{in-the-money}, \textit{at-the-money} and
\textit{out-of-the-money} contracts. Furthermore, a 
maturity ranging from 1 month to 1 year, initial underlying
prices $S_0^{(1)} =100 $ dollars/barrel, $S_0^{(2)}=2 $
dollars/gallon with an interest rate of $r=2\%$ are 
considered.

In Figure \ref{fig:wti} the series of daily future prices of
WTI in US dollar per barrel, from June 2104- July 2020 is
shown, while in Figure  \ref{fig:gaso} the series of future RBOB prices 
(in US dollar per gallon) over the same period can be 
observed. A rare negative price for WTI futures is observed 
on April 20th, 2020. We have  eliminated this observation to 
be able to work with log-returns. It is worth noticing that
the volatilities in the prices of both commodities have
increased after March 2020 due to the COVID-19 pandemic. 
Data have been taken from https://ca.investing.com. 
In \textit{crack} contracts  $c=\frac{1}{42}$ 
barrels/gallon.

\begin{figure}[ht]
    \begin{minipage}[b]{0.6\textwidth} 
    \includegraphics[scale=.25]{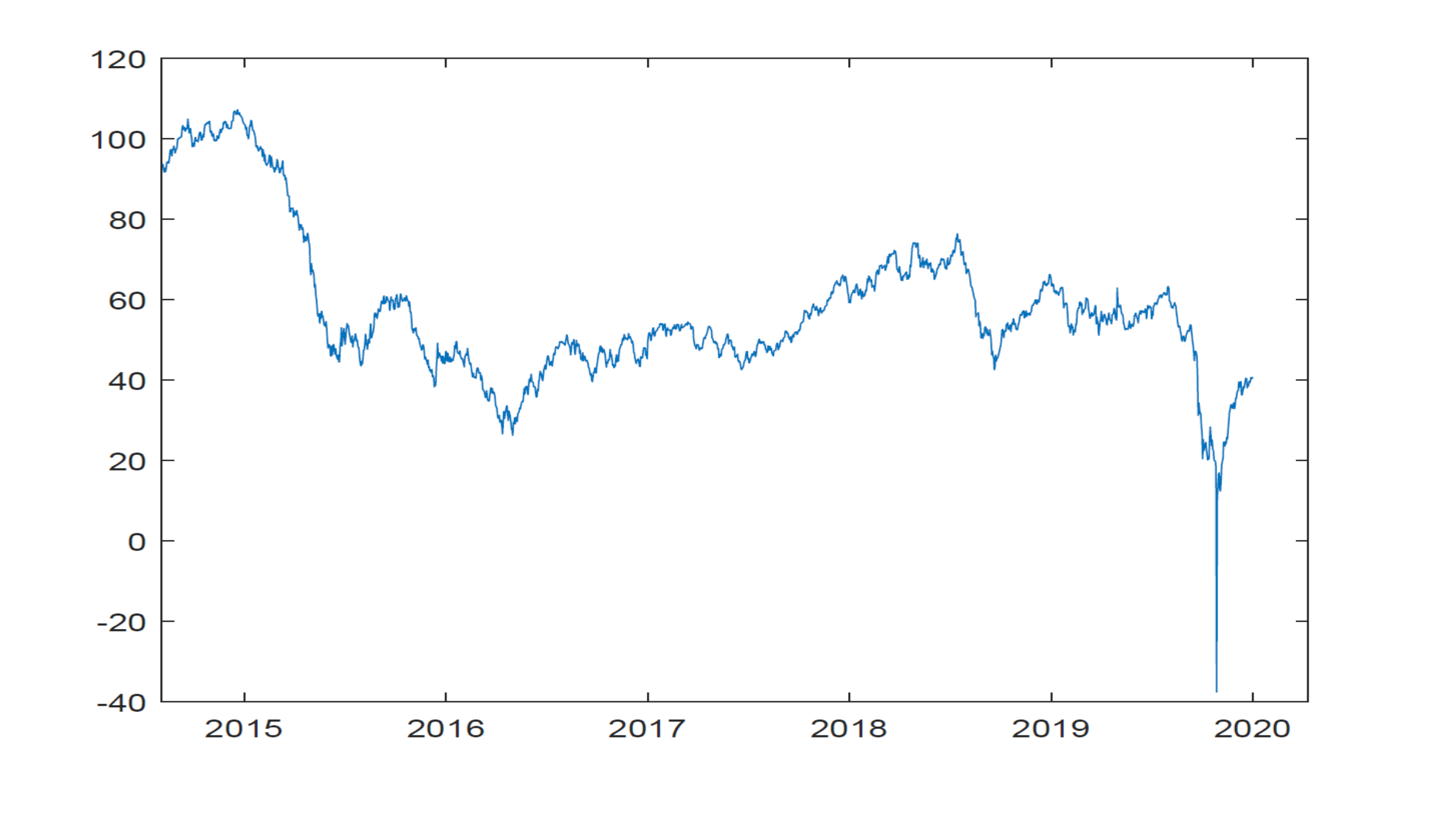}
    \caption{ series of daily future prices of WTI in US dollar per barrel, from June 2104- July 2020}
    \label{fig:wti}
    \end{minipage}
    \begin{minipage}[b]{0.5\textwidth}
    \includegraphics[scale=.25]{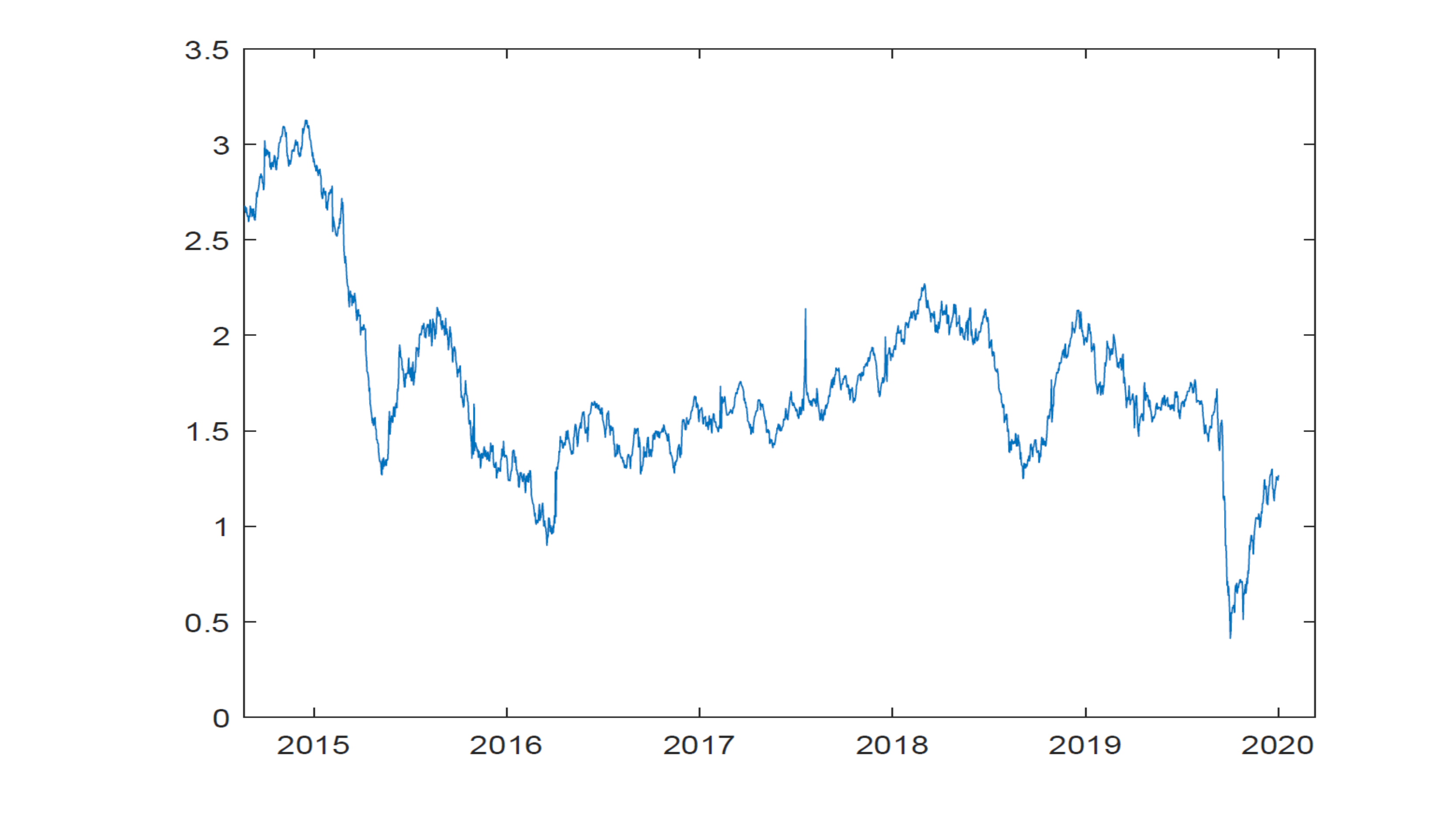}
    \caption{Series of future RBOB prices(in US dollar per gallon) over the same period }
    \label{fig:gaso}
    \end{minipage}
\end{figure}

In table \ref{Tab:Table3} the first four empirical moments appear for both RBOB and WTI as well as their log-returns. Volatilities have a similar range.  Specifically, annualized volatility of log- return WTI prices is 48.27 \% as for RBOB is 54.30 \%.
\begin{table}[ht]
    \centering
    \begin{tabular}{ |c||c||c||c||c|  }
     \hline
      Data set & Mean & St. dev. & Skewness & Kourtosis\\ \hline 
    RBOB prices &1.7408  &  0.4767 &   0.8071 &   4.1574 \\ \hline 
    RBOB log-returns  &  -0.0004  &    0.0299 &   -0.5568 &  23.2888 \\ \hline 
     WTI price   &  57.7983  & 18.4973  &  0.9846 &   3.8519  \\ \hline 
     WTI log-returns & -0.0005 &  0.0337 &  -3.0679 &  79.3556 \\ \hline
\end{tabular}
\caption{First four moments of RBOB and WTI prices and their log-returns}
\label{Tab:Table3}
\end{table}
 The WTI series exhibits a highly skewed distribution. Kurtosis of both series indicate the presence of heavy-tailed distributions. It is three times higher in WIT. As expected both assets are highly correlated. The correlation between prices is about 0.96, while the correlation between log-returns is 0.65.
\subsection{Double Merton jump-diffusion model implementation}
Parameters in the bivariate jump-diffusion Merton model with common and idiosyncratic jumps introduced in example \ref{Ex:1}.

\begin{table}[ht]
    \centering
    \begin{tabular}{ |c||c|  }
     \hline
    Parameters & Value \\ \hline 
  Interest rate & $ r=0.02$ \\ \hline 
   Diffusion & $\sigma_1=0.7025, \sigma_2=0.5356, \rho_B=0.5364$ \\ \hline 
  idiosyncratic   jump sizes & $\mu^{(1)}_J=0, \mu^{(2)}_J=0$\\ 
     & $\sigma^{(1)}_J=0.2808, \sigma^{(2)}_J=0.3528$ \\ \hline
 common jump sizes  &  $\mu_{0,1}=-0.0775, \mu_{0,2}= -0.0620$ \\ \hline
 & $\sigma_{0,1}=0.02, \sigma_{0,2}=0.01, \rho_{J}=0.30$\\ \hline
    Jump intensities & $\lambda_0=3, \lambda_1=2, \lambda_2=2 $  \\ \hline
    initial prices & $S_0^{(1)} =100 , S_0^{(2)}=2$\\ \hline
\end{tabular}
\caption{Parameters in the jump-diffusion model}
\label{Tab:Table4}
\end{table}
 Unfortunately, the number of parameters in the model is exceedingly large for most  known estimation methods. Therefore, we rather arbitrarily impose values to some of the parameters, while using a Generalized Method of Moment (GMM) approach to calibrate the remaining ones. We match first, second moments and correlations for both underlying assets as shown in equation (\ref{eq:mathcmom}). It leads to an optimization problem with bound constraints in the variables. Additional constraints imposed  by the EMM choice are reflected in equation (\ref{eq:rnconst}).    The results can be seen  are shown in table \ref{Tab:Table4}.
 Notice the jumps are not directly observable, making the estimation or calibration of the parameters a challenging task. Hence for $j=1,2$:
\begin{eqnarray}\nonumber
   \frac{\bar{\Delta Y^{(j)}}}{\Delta t}&=& \lambda_j \mu^{(j)}_J+ \lambda_0  \mu_{0,j} \\ \nonumber
   \frac{S^2(\Delta Y^{(j)})}{\Delta t}+\frac{((\bar{\Delta Y^{(j)}})^2}{\Delta t} &=& \sigma^2_j  +\lambda_j (\sigma^{(j)}_J)^2+\lambda_0  \sigma^2_{0,j}+ \lambda_j (\mu^{(j)}_J)^2+ \lambda_0  \mu^2_{0,j} \\ \label{eq:mathcmom}
   corr(\Delta Y^{(1)},\Delta Y^{(2)})&=& \frac{\rho_B \sigma_1 \sigma_2 \Delta t+ \rho_{J}\sigma_{0,1}\sigma_{0,2}\Delta t}{S(\Delta Y^{(1)}) S(\Delta Y^{(2)})}
\end{eqnarray}
where $\bar{Y}$ and $S^2(Y)$ are respectively the sample mean and the sample variance of the vector $Y$.
The time interval, measured in year units, is $\Delta t=\frac{1}{310}$ corresponding to 310 trading days.
From the condition $\Psi_{Y^{(i)}}(-i)=r$ and the drift $b_j=\lambda_j \mu^{(j)}_J+ \lambda_0  \mu_{0,j}$
 we add another two constrains to the parameters:
 \begin{eqnarray}\label{eq:rnconst}
 && \frac{\bar{\Delta Y^{(j)}}}{\Delta t}=  r-\frac{1}{2}\sigma^2_j-\lambda_j(\exp(\mu^{(j)}_J+ \frac{1}{2} (\sigma^{(j)}_J)^2)-1)- \lambda_0(\exp(\mu^{(j)}_{0,J} + \frac{1}{2} (\sigma^2_{0,j})-1).
 \end{eqnarray}
Next, we run the scheme \eqref{eq:thetasheme} with $\theta = 1/2$ to obtain approximated solutions $U^N$ 
in  meshes $\mathcal{T}^h(\Omega_b)$ of norm $h=L/2^N$ 
for $N=4,5,6,7,8$. Errors in norm $\|\cdot\|_{L^2}$ are
computed using a reference solution $U^{N_r}$ where $N_r=9$.
The BICGSTAB method is set to run until finding a solution with
tolerance to the residual relative error of $10^{-10}$. 
We are interested in values for both $S_1$ and 
$c\cdot S_2$ varying between $0$ and $3$ dollars, so the
effective domain is set $\Omega = \left[ -1.1 \,,\, 1.1\right]^2 \approx \left[log(0.3)\,,\,log(3)\right]$ while boundary  conditions (identical to the initial condition at each time) are imposed outside the outer domain $\Omega_b = \left[ -4, -4\right]^2$ to attenuate boundary error propagation to the interior of $\Omega$.
\begin{figure}[ht]
    \begin{minipage}[b]{0.5\textwidth} 
    \includegraphics[trim={1cm 0cm 1cm 1cm},clip,scale=.45]{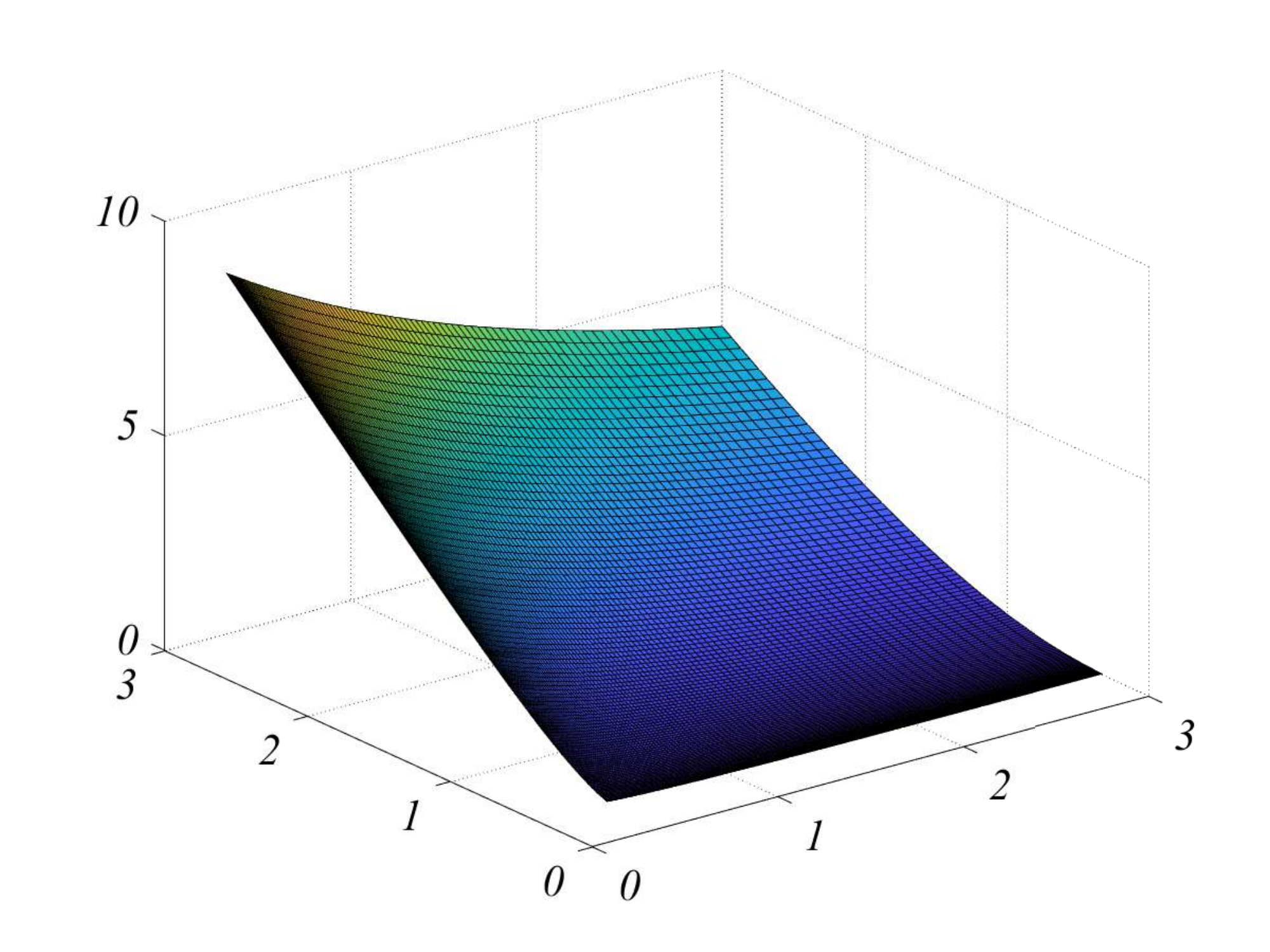}
    \caption{Reference solution $U^{N_r}$ with $N_r=10$. Time to maturity in years }
    \label{fig:sol}
    \end{minipage}
    \hspace{1cm}
    \begin{minipage}[b]{0.4\textwidth}
    \includegraphics[trim={0.9cm -2cm 1.9cm 1cm},clip,scale=.37]{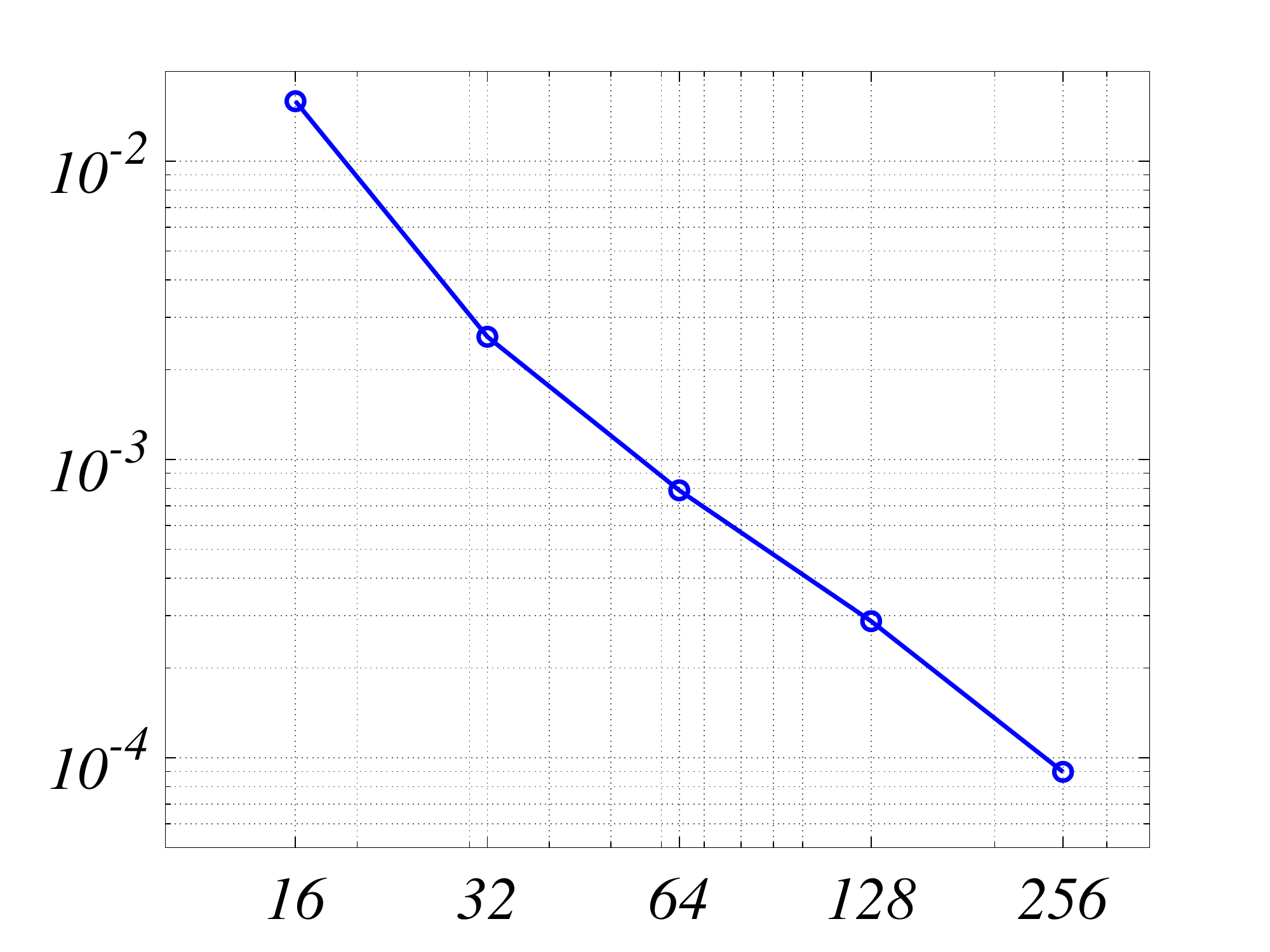}
    \caption{Log-log graphic of errors $\|U^N - U^{N_r}\|_{L^2}$ for $N=4,5,6,7,8$.}
    \label{fig:loglogMerton}
    \end{minipage}
\end{figure}

In  Table \ref{Tab:Table1} we can see the performance of the
algorithm at each grid of norm $h=L/2^N$. The first column
reflects the mesh step size. The second column  shows the magnitude of the error in  $L^2$-norm relative to the reference solution, while the third column provides the average number of iterations for time step of BICGSTAB. Finally, the fourth column shows the running time (in seconds) until a complete set of solutions $C(t,x)$ for values of $ t \in [0,T]$ in the mesh.
\begin{table}[ht]
    \centering
    \begin{tabular}{ |c||c||c||c|  }
     \hline
     \multicolumn{4}{|c|}{Performance} \\
     \hline
    $N$ & $\frac{\|U^{N_r} - U^N \|_{L^2}}{\|U^{N_r}\|_{L^2}}$ &  $\#$ it. by time step & time (sec.)  \\
     \hline
     $4$   & $1.5888  \cdot 10^{-2}$  & $5.8$   &   $0.6$    \\
     $5$  & $2.5788  \cdot 10^{-3}$   & $5.1$   &   $2.1$   \\
     $6$  & $7.8801  \cdot 10^{-4}$  & $4.2$   &   $13.5$  \\
     $7$  & $2.8699  \cdot 10^{-4}$  & $3.5$   &   $52.2$  \\
     $8$  & $8.9660  \cdot 10^{-5}$  & $2.8$   &   $241.8$ \\
     \hline
\end{tabular}
\caption{Performance of the algorithm at each grid of norm $h=L/2^N$. The second column: magnitude of the error in norm $\|\cdot\|_{L^2}$. Third column: Average number of iterations for time step of BICGSTAB. Fourth column: Running time in seconds.}
\label{Tab:Table1}
\end{table}
Figure \ref{fig:sol} shows the graphic of the reference solution $U^{N_r}$ with $N_r=10$ and Figure \ref{fig:loglogMerton} shows the decreasing behavior of the error in norm $L^2$ of the approximated solution $U^N$
with respect to the reference solution $U^{N_r}$ through the mesh refinement.

The $c^{(0)}_{FF}$ pre-conditioner performs very well since the number of iterations for time step does not grow through the mesh refinement, on the contrary, it seems to asymptotically decrease when $h\rightarrow 0$ which is a very desirable feature. The convergence rate was $1.8106$, very close to the  theoretical second order of the Crank-Nicholson scheme. The computational time grows linearly with the mesh refinement.

\subsection{Time-change model with Gamma subordinator implementation}
We set the interest rate to $r=0.02$ and  the loading parameters to $d_1=d_2=1$. Also, we consider the case $\mu_1= \mu_j$ to simplify calculations. Notice that  the remaining parameters to be estimated come from the three subordinators, the drift  parameters $\mu_1$ and $\mu_2$ and the volatility parameters $\sigma_{1,r}$ and $ \sigma_{2,r}$.

Once again we implement a GMM approach letting to a least square constrained minimization problem, this time in two steps. First by  matching the first four empirical and theoretical moments for the WTI asset prices taking into account risk neutral constraints, then matching the correlation to compute $/mu_2$ and the first two moments of the RBOB asset to compute the remaining parameters. The first matching is given by the non-linear equations
 \begin{eqnarray*}
  m^{(j)}_1  &=& \mu_j E_{j,1} \\
    m^{(j)}_2 &=& \mu^2_j E_{j,2}+ \sigma^2_{r,j} E_{j,1}\\ 
   m^{(j)}_3 &=& \mu^3_j E_{j,3}+ 3 \mu_j \sigma^2_{r,j}E_{j,2}\\ 
    m^{(j)}_4 &=& \mu^4_j E_{j,4}+6 \mu^2_j \sigma^2_{r,j}E_{j,3} + 3 \sigma^4_{r,j} E_{j,2} \\
   m_{12} &=& \mu_1  \mu_2 E[(\Delta L^{(0)}_t)^2],
 \end{eqnarray*}
where $m^{(j)}_k,\; j=1,2\;\; k=1,2,3,4$ are the empirical moments of log-returns  of both commodities and $m_{12}$   is its empirical mixed moment. In addition  the risk-neutral frame imposes the additional constraints
    \begin{eqnarray} \nonumber
    && -\alpha_0 \log \left( 1-\frac{\mu_j}{\beta_0}-\frac{\sigma^2_{r,j}}{2 \beta_0} \right)
 -\alpha_j \log \left( 1-\frac{\mu_j d_j}{\beta_j}-\frac{\sigma^2_{r,j} }{2 \beta_j} \right)=r,\;\;\ 
\end{eqnarray}
and the inequality constraint $0 <  \mu_j+\frac{\sigma^2_{r,j} }{2} < min(\beta_0,\beta_j) $. It leads to a  system of six variables with six equations. For convenience we re-parametrize it to new variables $x_1=\frac{\alpha_0}{\beta_0}, x_2=\frac{\alpha_j}{\beta_j}, x_3=\beta^{-1}_0, x_4=\beta^{-1}_j, x_5=\mu_j, x_6=\sigma^2_{j,r}$. Results of the parameter estimation can be viewed in table \ref{Tab:Table5}.
\begin{table}[ht]
    \centering
    \begin{tabular}{ |c||c|  }
     \hline
    Parameters & Value \\ \hline   
  Interest rate & $ r=0.02$ \\ \hline 
  WTI subordinator  & $\alpha_1=0.7, \beta_1=0.7 $ \\ \hline 
 RBOB subordinator & $\alpha_2=0.8, \beta_2=0.8 $ \\ \hline 
 Common subordinator & $\alpha_0=0.5   , \beta_0=0.5 $ \\ \hline
  Loading parameters   & $d_1=1, d_2=1$ \\ \hline
parameter drift & $\mu_1=-0.0673, \mu_2=-0.050701$\\ \hline
parameter Brownian volatility & $\sigma_{r,1}=0.4633, \sigma_{r,2}=0.2236$ \\ \hline
Initial prices & $S_0^{(1)}=100, S_0^{(2)}=2$ \\ \hline
\end{tabular}
\caption{Parameters in the time-changed two dimensional model with Gamma subordinator}
\label{Tab:Table5}
\end{table}
Multiple initial values have been tested to avoid local minima. 

Same as in the previous example, we  we run the scheme \eqref{eq:thetasheme} with $\theta = 1/2$ to obtain approximated solutions $U^N$ in 
meshes $\mathcal{T}^h(\Omega_b)$ of norm $h=L/2^N$ 
for $N=4,5,6,7,8$.
Errors in norm $\|\cdot\|_{L^2}$ are computed using a reference solution $U^{N_r}$ where $N_r=9$. The BICGSTAB method was set run until find a solution with tolerance to the residual relative error of $10^{-10}$. Values for both $S_1$ and $c\cdot S_2$ are varying between $0$ and $3$ dollars, so the effective domain is $\Omega = \left[ -1.1 , 1.1\right]^2 \approx \left[log(0.3),log(3)\right]^2$ while boundary conditions (identical to the initial condition at each time) are imposed outside the outer domain $\Omega_b = \left[ -4, -4\right]^2$ to attenuate boundary error propagation to the interior of $\Omega$.

The graphic of the reference solution $U^{N_r}$ with $N_r=10$ is shown in Figure \ref{fig:solGanma} while the graphic in Figure \ref{fig:loglogGanma} shows the decreasing behavior of the error in norm $L^2$ of the approximated solution $U^N$
with respect to the reference solution $U^{N_r}$ through the mesh refinement. In  Table \ref{Tab:TableGamma} we can see the performance of the algorithm at each grid of norm $h=L/2^N$. The first column reflects the mesh step size. The second column  shows the magnitude of the error in  $L^2$-norm relative to the reference solution, while the third column provides the average number of iterations for time step of BICGSTAB. The fourth column shows the running time (in seconds) until a complete set of solutions $C(t,x)$ for values
of $ t \in [0,T]$ in the mesh.
\begin{figure}[ht]
    \begin{minipage}[b]{0.5\textwidth} 
    \includegraphics[trim={1cm 0cm 1cm 1cm},clip,scale=.45]{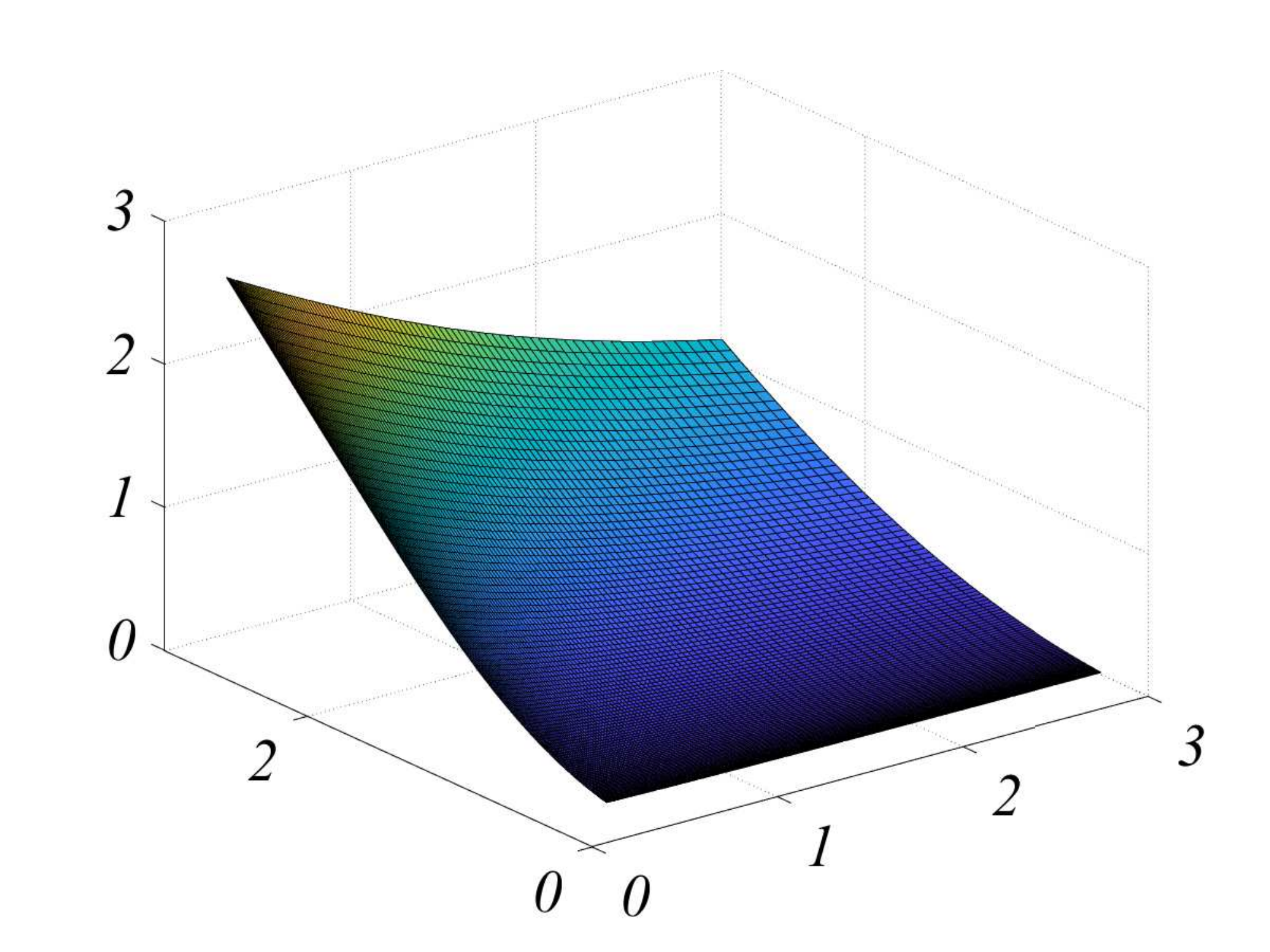}
    \caption{Reference solution $U^{N_r}$ with $N_r=10$. Time to maturity in years }
    \label{fig:solGanma}
    \end{minipage}
    \hspace{1cm}
    \begin{minipage}[b]{0.4\textwidth}
    \includegraphics[trim={0.9cm -2cm 1.9cm 1cm},clip,scale=.37]{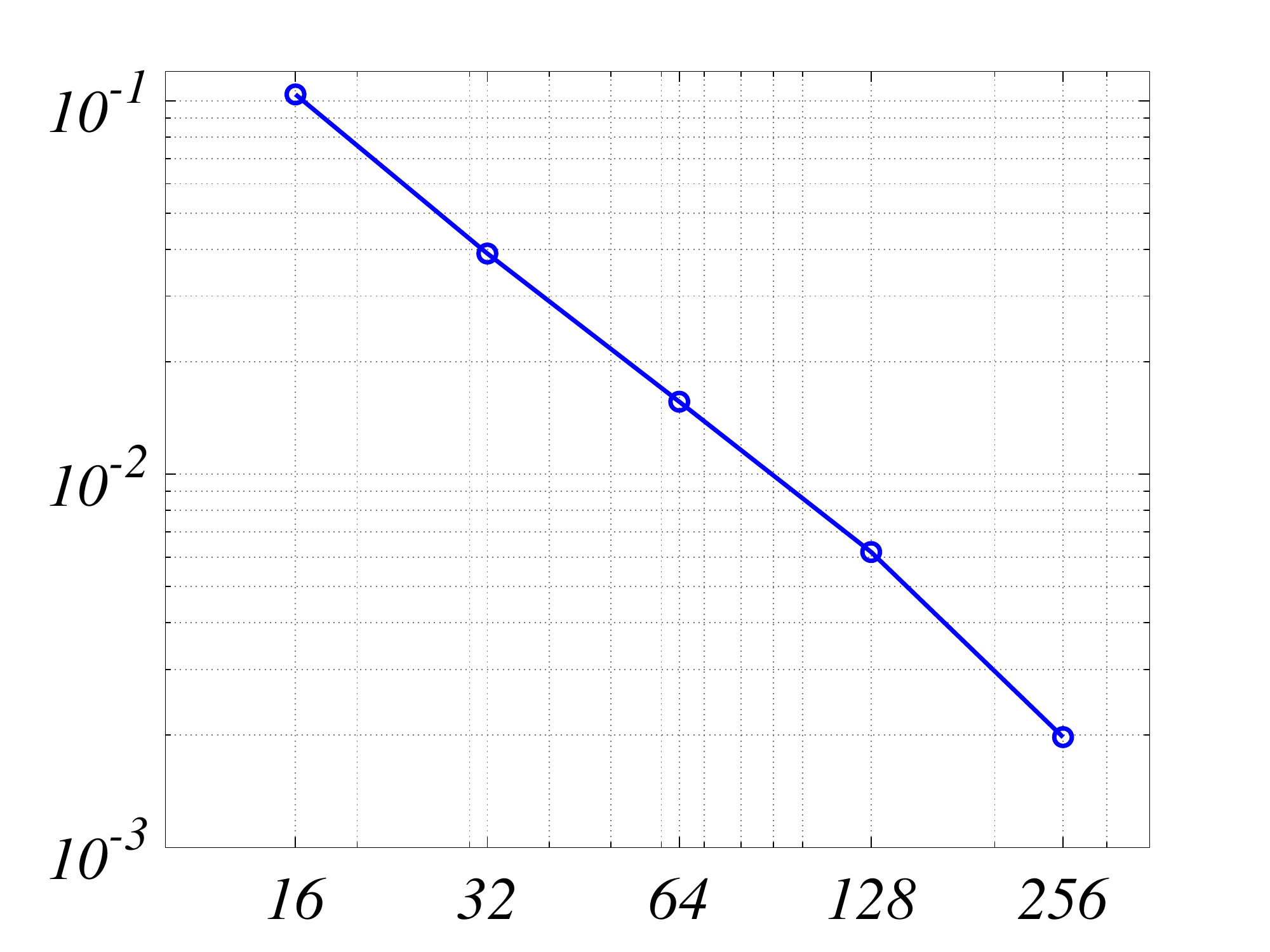}
    \caption{Log-log graphic of errors $\|U^N - U^{N_r}\|_{L^2}$ for $N=4,5,6,7,8$.}
    \label{fig:loglogGanma}
    \end{minipage}
\end{figure}
\begin{table}[ht]
    \centering
    \begin{tabular}{ |c||c||c||c|  }
     \hline
     \multicolumn{4}{|c|}{Performance} \\
     \hline
    $N$ & $\frac{\|U^{N_r} - U^N \|_{L^2}}{\|U^{N_r}\|_{L^2}}$ &  $\#$ it. by time step & time (sec.)  \\
     \hline
     $4$   & $1.0412  \cdot 10^{-1}$  & $5.4$   &   $0.5$    \\
     $5$  & $3.9017  \cdot 10^{-2}$   & $4.9$   &   $1.8$   \\
     $6$  & $1.5644  \cdot 10^{-2}$  & $4.0$   &   $12.9$  \\
     $7$  & $6.1894  \cdot 10^{-3}$  & $3.2$   &   $49.3$  \\
     $8$  & $1.9759  \cdot 10^{-3}$  & $2.4$   &   $233.2$  \\
     \hline
\end{tabular}
\caption{Performance of the algorithm at each grid of norm $h=L/2^N$. The second column: magnitude of the error in norm $\|\cdot\|_{L^2}$. Third column: Average number of iterations for time step of BICGSTAB. Fourth column: Running time in seconds.}
\label{Tab:TableGamma}
\end{table}

Again, the $c^{(0)}_{FF}$ pre-conditioner performs very well exhibiting an average value of iterations for time step that not only doesn't grow through the mesh refinement but on the contrary, it seams to asymptotically decreases when $h\rightarrow 0$ which is a very desirable feature. The convergence rate was $1.47$, a little bit lower than the  theoretical second order of the Crank-Nicholson scheme.

\section{Conclusions}

The proposed strategy  proved to be efficient for the valuation of spread contracts through the resolution of the associated PIDE. The symbol method allowed the efficient construction of the stiffness matrix, based on the use of the FFT, while the iterative method BICGSTAB for the solution of non-symmetric linear systems accompanied by the circular preconditioner facilitated the implementation of an implicit scheme for the temporal evolution by allowing the efficient resolution of the BTTB system associated to the Galerkin discretization. 

This strategy can  be extended to three dimensions although we must bear in mind that the computational cost grows exponentially with the dimensionality. Other terminal conditions could be considered using this strategy as well as barrier problems.  

Unfortunately, the symbol method can not be directly used with most classic finite element approximation spaces since their basis are not smooth enough to efficiently access their symbols through FFT, for example $Q^r$ spaces has polynomial interpolation basis are in $H^1_0(\Omega)$. 
This issue can be addressed by substituting the basis for
a mollified version of it as done in \cite[Section 4.4.1]{gass2018flexible}.
Even if we do that, in most of the cases the basis are obtained through translations of more than one parent function, which leads us to the resolution of a linear  system with a blocked matrix where each block is BTTB matrix. The system is also numerically tractable by considering an approach as in the present work but implementation becomes more complicated.

\section{Acknowledgements}
This research has been funded by NSERC and Fields Institute.


\appendix
\section{Appendix}

\subsection{Parameter choice for spread options}
\label{A:parameter}
Let us consider a spread option. The payoff has the form
\begin{equation}
    h(x,y) = \left(a_2\text{e}^y-a_1\text{e}^x-K \right)^+
\end{equation}
where $a_1$ and $a_2$ are positive constants and the function
$(r)^x$ is the maximum between zero and $r$. We want to know
what the values of $\eta$ should be for $h\in L^1_\eta(\mathbb{R}^2)$ and $h\in L^2_\eta(\mathbb{R}^2)$. 

There is a function $y:\mathbb{R}\rightarrow\mathbb{R}$  that separates the plane in two sets, the points at which $h$ is greater than zero and the points at which $h$ is zero. Such a curve is as follows
\begin{equation*}
y(x) = log(\frac{a_1\text{e}^x + K}{a_2})
\end{equation*}
So we can write
\begin{align*}
    \|h\|_{L^1_\eta(\mathbb{R}^2)} &= \int_{-\infty}^\infty\int^\infty_{y(x)} |a_2\text{e}^y-a_1\text{e}^x-K| \text{e}^{\eta_1x}\text{e}^{\eta_2y}\text{d}y\text{d}x\\
   &\leq \int_{-\infty}^\infty\int^\infty_{y(x)} a_2\text{e}^{(1+\eta_2)y}\text{e}^{\eta_1x}\text{d}y\text{d}x+
   \int_{-\infty}^\infty\int^\infty_{y(x)} a_1\text{e}^{(1+\eta_1)x}\text{e}^{\eta_2y}\text{d}y\text{d}x+
   \int_{-\infty}^\infty\int^\infty_{y(x)} K\text{e}^{\eta_1x}\text{e}^{\eta_2y}\text{d}y\text{d}x\\ 
   &\leq a_2\int_{-\infty}^\infty\int^\infty_{y(x)} \text{e}^{(1+\eta_2)y}\text{e}^{\eta_1x}\text{d}y\text{d}x+
   a_1\int_{-\infty}^\infty\int^\infty_{y(x)} \text{e}^{(1+\eta_1)x}\text{e}^{\eta_2y}\text{d}y\text{d}x+
   K\int_{-\infty}^\infty\int^\infty_{y(x)} \text{e}^{\eta_1x}\text{e}^{\eta_2y}\text{d}y\text{d}x\\
   &\leq a_2\int_{-\infty}^\infty\text{e}^{\eta_1x}\int^\infty_{y(x)} \text{e}^{(1+\eta_2)y}\text{d}y\text{d}x+
   a_1\int_{-\infty}^\infty \text{e}^{(1+\eta_1)x}\int^\infty_{y(x)} \text{e}^{\eta_2y}\text{d}y\text{d}x+
   K\int_{-\infty}^\infty\text{e}^{\eta_1x}\int^\infty_{y(x)} \text{e}^{\eta_2y}\text{d}y\text{d}x
\end{align*}
Imposing $\eta_2 < -1$ all seconds integrals exist as a function of $x$. Note that $(-1-\eta_2)$ and $-\eta_2$ are both positive quantities so the following integrals remain positive.
\begin{align*}
&\leq a_2\int_{-\infty}^\infty\text{e}^{\eta_1x} \frac{\text{e}^{(1+\eta_2)y(x)}}{(-1-\eta_2)} \text{d}x+
a_1\int_{-\infty}^\infty \text{e}^{(1+\eta_1)x} \frac{\text{e}^{\eta_2y(x)}}{-\eta_2}\text{d}x +
K\int_{-\infty}^\infty\text{e}^{\eta_1x} \frac{\text{e}^{\eta_2y(x)}}{-\eta_2}\text{d}x\\
&\leq \frac{a_2}{(-1-\eta_2)}\int_{-\infty}^\infty\text{e}^{\eta_1x} \left(\frac{a_1\text{e}^x + K}{a_2}\right)^{1+\eta_2} \text{d}x+
\frac{a_1}{-\eta_2}\int_{-\infty}^\infty \text{e}^{(1+\eta_1)x} \left(\frac{a_1\text{e}^x + K}{a_2}\right)^{\eta_2}\text{d}x +\\
& \quad\quad\quad\quad\quad\quad\quad\quad\quad\quad\quad\quad+ \frac{K}{-\eta_2}\int_{-\infty}^\infty\text{e}^{\eta_1x} \left(\frac{a_1\text{e}^x + K}{a_2}\right)^{\eta_2}\text{d}x\\
&\leq \frac{a_2^{-\eta_2}}{(-1-\eta_2)}\int_{-\infty}^\infty \left(a_1\text{e}^{A_1x} + K\text{e}^{A_2x}\right)^{1+\eta_2} \text{d}x+
\frac{a_1a_2^{-\eta_2}}{-\eta_2}\int_{-\infty}^\infty  \left(a_1\text{e}^{B_1x} + K\text{e}^{B_2x}\right)^{\eta_2}\text{d}x +\\
& \quad\quad\quad\quad\quad\quad\quad\quad\quad\quad\quad\quad +\frac{K a_2^{-\eta_2}}{-\eta_2}\int_{-\infty}^\infty\left(a_1\text{e}^{C_1x} + K\text{e}^{C_2x}\right)^{\eta_2}\text{d}x,
\end{align*}
where
\begin{align*}
   A_1=1+\eta_1/(1+\eta_2) &,& &A_2=\eta_1/(1+\eta_2),& &B_1=1+ (1+\eta_1)/\eta_2, && B_2=(1+\eta_1)/\eta_2 \\
   && &C_1=1+\eta_1/\eta_2,& &C_2=\eta_1/\eta_2 
\end{align*}
In order the integrand to have exponential decay the pairs
$(A_1,A_2)$, $(B_1,B_2)$ and $(C_1,C_2)$ need to alternate signs which immediately means that $\eta_1$ has to be greater than zero because if not all signs are negative. Now, setting $\eta_1>0$ we have $A_2<0$, $B_2<0$ and $C_2<0$ and we need to ask for $A_1>0$, $B_1>0$ and $C_1>0$ which lead us to $\eta_1<-\eta_2-1$. Putting all together we need to impose
$\eta\in(0,a-1)\times(-\infty,-a)$ for any $a>1$ in order $h\in L^1_\eta(\mathbb{R}^2)$. If we repeat the process for $\|\cdot\|_{L^2_\eta(\mathbb{R}^2)}$ we get the same therms as before (up to a constant) and three new terms from $\left|a_2\text{e}^y-a_1\text{e}^x-K\right|^2\leq a_1^2\text{e}^{2y}+a_2^2\text{e}^{2x} + 2a_1a_2\text{e}^{x}\text{e}^{y}+\dots$ so we need the following integrals to exist
\begin{align}
\int_{-\infty}^\infty\text{e}^{\eta_1x}\int^\infty_{y(x)} \text{e}^{(2+\eta_2)y}\text{d}y\text{d}x&, 
&\int_{-\infty}^\infty \text{e}^{(2+\eta_1)x}\int^\infty_{y(x)} \text{e}^{\eta_2y}\text{d}y\text{d}x&,
&\int_{-\infty}^\infty\text{e}^{(1+\eta_1)x}\int^\infty_{y(x)} \text{e}^{(1+\eta_2)y}\text{d}y\text{d}x
\label{eq:integrals}
\end{align}
using the same arguments as for $L^1_\eta(\mathbb{R}^2)$, we have from the first integral in \eqref{eq:integrals} that $\eta_2<-2$. We obtain another three pairs $(D_1,D_2)$, $(E_1,E_2)$ and $(F_1,F_2)$ that need to alternate signs, where 
\begin{align*}
   D_1=1+\eta_1/(2+\eta_2) &,& &D_2=\eta_1/(2+\eta_2),& &E_1=1+ (2+\eta_1)/\eta_2, && E_2=(2+\eta_1)/\eta_2 \\
   && &F_1=1+(1+\eta_1)/(1+\eta_2),& &F_2=(1+\eta_1)/(1+\eta_2).
\end{align*}
The choice $ \eta\in(0,a-2)\times(-\infty,-a)$ for any $a>2$ warranties that $h\in L^2_\eta(\mathbb{R}^2)$ in fact, it also warranties that $h\in L^1_\eta(\mathbb{R}^2)\cap L^2_\eta(\mathbb{R}^2)$.

\subsection{Continuity}
\label{A:Continuity}

Let us denote by $m$ and $M$ the minimum and maximum eigenvalues of $\Sigma_B$ respectively, from \eqref{eq:chfcppm} we have that
\begin{align*}
  \left| A(u) \right| &= \left| i u b^{\mathcal{Q}} + \frac{1}{2}u \Sigma_B u^T +  \sum_{j=1}^2 \lambda_j (\varphi_{ X^{(j)}}(u_j)-1) + \lambda_0 (\varphi_{X_0}(u)-1) \right| \\
  & \leq  \left|b^{\mathcal{Q}}\right|\left| u \right|+ \frac{1}{2}M\left| u \right|^2 + 2\left(\lambda_0+\lambda_1+\lambda_2 \right)\\
  & \leq \max\bigg\lbrace\left|b^{\mathcal{Q}}\right| \,,\,\frac{1}{4}M \,,\, 2\left(\lambda_0+\lambda_1+\lambda_2 \right)\bigg\rbrace \left(1+2|u| +|u|^2\right)\\
  &\leq C_1(1+|u|)^2,
\end{align*}
and (A2) is satisfied for $\alpha=2$ equal to the maximum in the thirst line.

\subsection{G{\aa}rding condition}
\label{A:Garding}

The \textit{G{\aa}rding condition} (A3) is also satisfied for Example \ref{Ex:1}. 
\begin{align*}
    \mathbb{R}(A(u)) &= \frac{1}{2}u \Sigma_B u^T + \sum_{j=1}^2\lambda_j\left(\text{e}^{-\frac{1}{2}  (\sigma_J^{(j)})^2 u^2_j}
    \cos (\mu^{(j)}_J u_j)-1\right) + \lambda_0\left(\text{e}^{-\frac{1}{2}u \Sigma_{0,J} u^T} \cos  (\mu_{0,J}\cdot u^T) -1 \right)\\
    &\geq \frac{1}{2}m|u|^2 - 2\left(\lambda_0 + \lambda_1 + \lambda_2 \right)\\
    &\geq \left(\sqrt{2m}|u| +  \frac{1}{2}m\right) - \left(\sqrt{2m}|u| +  \frac{1}{2}m\right) + \frac{1}{2}m|u|^2 - 2\left(\lambda_0 + \lambda_1 + \lambda_2 \right)\\
    &\geq \frac{1}{2}m (1+2|u|+|u|^2) - \max\bigg\lbrace2\left(\lambda_0 + \lambda_1 + \lambda_2 \right) + \frac{1}{2}m \,,\,  \sqrt{2m}\bigg\rbrace (1+|u|)\\
    &\geq C_2 (1+|u|)^2 - C_3 (1+|u|)^1,
\end{align*}
and (A3) is satisfied for index $\alpha = 2$ and $\beta =1$.

\bibliography{pidespreads1} 
\bibliographystyle{ieeetr}

\end{document}